\documentclass[11pt,letterpaper,reqno]{amsart}
\usepackage[includehead,includefoot,margin=25mm]{geometry}
\usepackage{amsfonts}
\usepackage{amsmath}
\usepackage{amssymb}
\usepackage{amsthm}
\usepackage{thmtools}
\usepackage{times}
\usepackage{color}
\usepackage{enumerate}
\usepackage[T1]{fontenc}
\usepackage[utf8]{inputenc}
\usepackage{graphicx}
\usepackage{nicefrac}
\usepackage{xcolor}
\usepackage[all]{xy}
\usepackage{comment}
\usepackage{hyperref}
\hypersetup{
    colorlinks=true,
    linkcolor=blue,
    citecolor=blue,
    filecolor=blue,
    urlcolor=blue
}
\usepackage{float}
\usepackage[normalem]{ulem}
\usepackage{pgfplots}
\pgfplotsset{compat=1.15}
\usepackage{mathrsfs}
\usetikzlibrary{arrows}

\usepackage[nameinlink,capitalize,noabbrev]{cleveref}
\declaretheorem[name=Theorem,numberwithin=section]{theorem}
\declaretheorem[name=Theorem,numbered=no]{theorem*}
\declaretheorem[sibling=theorem,name=Lemma]{lemma}
\declaretheorem[sibling=theorem,name=Proposition]{proposition}
\declaretheorem[sibling=theorem,name=Corollary]{corollary}
\declaretheorem[sibling=theorem,name=Definition,style=definition]{definition}
\declaretheorem[sibling=theorem,name=Example,style=definition]{example}
\declaretheorem[sibling=theorem,name=Remark,style=remark]{remark}

\newcommand{\XX}{\mathbf{x}}

\def\x{\overline{x}}
\def\y{\overline{y}}

\newcommand{\KK}[0]{\ensuremath{\mathbf{k}}}

\newcommand{\ZZ}[0]{\ensuremath{\mathbb{Z}}}
\newcommand{\GA}[0]{\ensuremath{\mathbb{G}_{\mathrm{a}}}}
\newcommand{\GM}[0]{\ensuremath{\mathbb{G}_{\mathrm{m}}}}
\newcommand{\AF}[0]{\ensuremath{\mathbb{A}}}

\newcommand{\RR}[0]{\ensuremath{\mathbb{R}}}

\newcommand{\spec}[0]{\ensuremath{\operatorname{Spec}}}

\newcommand{\supp}[0]{\ensuremath{\operatorname{supp}}}

\newcommand{\Aut}[0]{\ensuremath{\operatorname{Aut}}}

\newcommand{\rank}[0]{\ensuremath{\operatorname{rank}}}

\newcommand{\homo}[0]{\ensuremath{\operatorname{Hom}}}

\newcommand{\Der}[0]{\ensuremath{\operatorname{Der}}}

\newcommand{\quot}[1]{{\overline{#1}}}

\makeatletter
\newcommand*\bigcdot{\mathpalette\bigcdot@{.5}}
\newcommand*\bigcdot@[2]{\mathbin{\vcenter{\hbox{\scalebox{#2}{$\m@th#1\bullet$}}}}}
\makeatother

\begin{document}

\title[The Automorphism groups of zero-dimensional monomial algebras]{The Automorphism groups of zero-dimensional monomial algebras}

\author{Roberto D\'iaz}
\address{Departamento de Matemáticas, Facultad de Ciencias, Universidad de La Serena, Av. Juan Cisternas 1200, La Serena, Chile.
}%
\email{roberto.diazv1@userena.cl}

\author{Alvaro Liendo}
\address{Instituto de Matem\'atica y F\'isica, Universidad de Talca,
 Casilla 721, Talca, Chile.}%
\email{aliendo@utalca.cl}

\author{Gonzalo Manzano-Flores} %
\address{Instituto de Matem\'aticas, Facultad de Ciencias, Universidad de Valpara\'iso, Gran Breta\~{n}a 1111, Valpara\'iso, Chile.}
\email{gonzalo.manzano@uv.cl}

	\author{Andriy Regeta}
	\address{Dipartimento di Matematica ``Tullio Levi-Civita'',
Universit\`a di Padova, Via Trieste 63, I-35121 Padova}
\email{andriyregeta@gmail.com}

\date{\today}

\thanks{
 \mbox{\hspace{11pt}}{\it Key words}: monomial ideal, locally nilpotent derivation, automorphism group.\\
 \mbox{\hspace{11pt}} The first author is supported by the Fondecyt postdoc project 3230406. The second author was partially supported by Fondecyt project 1240101. The third author is supported by the Fondecyt postdoc project 3240191 and Concurso de Subvenci\'on a la Instalaci\'on en la Academia 2025 N°85250089. The fourth author is supported by DFG, project number 509752046. The first and second authors were also supported by ECOS-ANID (Grant Number ECOS230044).}

\begin{abstract}
A monomial algebra $B$ is defined as a quotient of a polynomial ring by a monomial ideal, which is an ideal generated by a finite set of monomials. In this paper, we determine the automorphism group of a monomial algebra $B$, under the assumption that $B$ is a finite-dimensional vector space over a field $\KK$ of characteristic zero. We achieve this by providing an explicit classification of the homogeneous nilpotent derivations on $B$. The main body of the paper addresses the more general case of semigroup algebras, with the polynomial ring being a particular case.
\end{abstract}

\maketitle


\section*{Introduction}

An affine semigroup $S$ is a finitely generated semigroup with an identity element that can be embedded into $\ZZ^n$. Let $\KK$ be a field of characteristic zero and $S$ an affine semigroup. The semigroup algebra $\KK[S]$ is defined as $\KK[S]=\bigoplus_{m \in S} \KK \cdot \XX^m$, where $\XX^m$ are new symbols satisfying the multiplication rules $\XX^m \cdot \XX^{m'} = \XX^{m+m'}$ and $\XX^0 = 1$. A common example of a semigroup algebra, which motivates our notation, is the polynomial ring $\KK[\XX] = \KK[x_1, \ldots, x_n]$ obtained with $S = \ZZ_{\geq 0}^n$, where $\XX^m$ corresponds precisely to monomials under the usual multi-index notation. By analogy, we refer to all symbols $\XX^m$ as monomials of $\KK[S]$. An ideal $I \subseteq \KK[S]$ is called monomial if it has a generating set composed of monomials. A monomial algebra $\KK[S]/I$ is the quotient of $\KK[S]$ by a monomial ideal $I$. A monomial algebra $\KK[S]/I$ is zero-dimensional if and only if its dimension as a $\KK$-vector space is finite.

The combinatorial nature of monomial algebras allows them to be identified with objects from other categories, making these algebras suitable for developing examples and verifying general theories. For instance, Stanley-Reisner rings, which are monomial algebras defined by square-free monomial ideals, are in bijective correspondence with abstract simplicial complexes \cite[Theorem 1.7]{MS05}. Stanley proved the upper bound conjecture for spheres using the Cohen-Macaulay structure of these rings \cite{S75}, which was established by Reisner \cite{R76}. In another example, a special class of monomial algebras known as discrete Hodge algebras was studied in \cite{CEC82}, enabling the examination of homogeneous coordinate rings of Grassmann varieties and their Schubert subvarieties.

\medskip

\noindent \emph{Our main goal in this paper is to compute the automorphism groups of zero-dimensional monomial algebras.}

\medskip

Our approach to computing the automorphism group of a zero-dimensional monomial algebra builds upon the modern developments arising from Demazure’s seminal work \cite{demazure1970sous}.
In particular, subsequent contributions such as \cite{C95,liendo2010affine,LL21, DL24} reinterpret ideas in the context of varieties over a ground field using the notion of locally nilpotent derivation.

A normal variety is called toric if it admits an algebraic torus action with an open orbit, and affine toric varieties are precisely given by $\spec \KK[S]$, where $S$ is an affine semigroup \cite{oda1983convex,fulton1993introduction,cox2011toric}. Although the geometric nature of smooth proper toric varieties is significantly different from that of zero-dimensional monomial algebras, both objects have similar combinatorial descriptions. This similarity allows us to apply some of the techniques from \cite{demazure1970sous} to our context. We refer the reader to \cite{LL21}, whose first author is the second named author of the present paper, for a concise modern account of Demazure's results.

For a $\KK$-algebra $B$, we denote by $\Der(B)$ the $\KK$-vector space of $\KK$-derivations $\partial \colon B \to B$. Additionally, if $I \subset B$ is an ideal, we denote by $\Der_I(B)$ the subspace of $\Der(B)$ consisting of $\KK$-derivations such that $\partial(I) \subset I$. All our derivations are $\KK$-derivations, so we will omit $\KK$ from the notation. A derivation $\partial \in \Der(B)$ is called locally nilpotent if for every $f \in B$ there exists $\ell \in \ZZ_{\geq 0}$ such that $\partial^\ell(f) = 0$, where $\partial^\ell$ denotes the $\ell$-th composition of $\partial$.

As a first step towards our goal of computing the automorphism groups of monomial algebras, we describe the set $\Der_I(\KK[S])$, where $S$ is an affine semigroup and $I$ is a monomial ideal (see \cref{decomposition-derivations}). This description has already been provided for the special case where the semigroup algebra $\KK[S]$ is a polynomial ring in \cite{B95,T09}. Consider an embedding $S \hookrightarrow M$ where $M \cong \ZZ^n$ such that the group generated by $S$ in $M$ is $M$ itself. The algebra $\KK[S]$ admits an $M$-grading, and since $I$ is $M$-graded, $\KK[S]/I$ inherits this $M$-grading. We make extensive use of these $M$-gradings throughout the paper, and in particular, we derive \cref{decomposition-derivations} straightforwardly from \cite{KLL15} by the second named author.

Let $S$ be an affine semigroup and $I$ be a monomial ideal. There is a natural map $\pi\colon\Der_I(\KK[S])\to\Der(\KK[S]/I)$ induced by the quotient. In \cref{th:surjective derivatio}, we show that $\pi$ is surjective for every monomial ideal $I$ if and only if $\KK[S]$ is a polynomial ring. A derivation $\overline{\partial}\colon \KK[S]/I \to \KK[S]/I$ is said to be liftable if there exists a derivation $\partial\colon \KK[S] \to \KK[S]$ such that $\pi(\partial) = \overline{\partial}$. In \cref{main-classification}, we provide a classification of the liftable locally nilpotent derivations on the monomial algebra $\KK[S]/I$ that are homogeneous with respect to the $M$-grading.

The interest in locally nilpotent derivations on an algebra $B$ stems from their one-to-one correspondence with additive group actions on $B$; see, for instance, \cite[Section~1.1.5]{F17}. Specifically, given a locally nilpotent derivation $\partial\colon B \to B$, we define the automorphism
\begin{align} \label{eq:exponential}
 \exp(\partial)\colon B \to B \quad \text{defined by} \quad f \mapsto \sum_{i=0}^\infty \frac{ \partial^i(f)}{i!}\,.
\end{align}
Now the  $\GA$-action on $B$ associated to $\partial$ is obtained as $\GA\to \Aut_\KK(B)$ given by $s\mapsto \exp(s\partial)$.

To describe our main result, we let $S = \ZZ_{\geq 0}^n$ so that $B = \KK[S]$ is a polynomial ring and we let $I$ be a monomial ideal. Without loss of generality, we may assume that $I$ is full, that is, no standard basis vector in $\ZZ^n$ is contained in $I$ (see \cref{def:full}). The next theorem collects structural properties of the automorphism group of a zero-dimensional monomial algebra. Our main result, \cref{th:aut}, offers a more detailed and comprehensive formulation.

\begin{theorem*}\label{th:main}
Let $S=\ZZ^n_{\geq 0}$, making $B=\KK[S]$ a polynomial ring, and let $I \subset B$ be a full monomial ideal such that $B/I$ is a zero-dimensional monomial algebra. If $G=\Aut_\KK(B/I)$, then the following statements hold:
\begin{enumerate}[$(i)$]
 \item The group $G$ is linear algebraic.
 \item The algebraic torus $T=\spec\KK[\ZZ^n]$ is a maximal torus of $G$.
 \item The connected component $G^0$ of $G$ is generated by the maximal torus and the images in $G^0$ of all the $\GA$-actions corresponding to homogeneous nilpotent derivations $\partial\colon B/I \to B/I$.
 \item The finite group $G/G^0$ is generated by the images of automorphisms $B/I \to B/I$ induced by semigroup automorphisms $S \to S$ that map $I$ to itself.
\end{enumerate}
\end{theorem*}

Items $(i)-(iii)$ describe general features of linear algebraic groups, providing the algebraic framework for our study. 
In particular, $(ii)$ is not automatic: we show that the centralizer of the torus $T$ in $\operatorname{Aut}_\KK(B/I)$ equals $T$ (see \cref{lemma: torus}). 
This equality relies essentially on the assumption that $S$ is the first octant and that $I$ is a full monomial ideal; outside this setting the statement may fail (see \Cref{example: non maximal torus}). 
The main novelty of our work, too technical to be discussed in this introduction, lies in providing an explicit version of $(iii)$ through the classification of homogeneous locally nilpotent derivations (\cref{main-classification}), together with $(iv)$.
Regarding $(iv)$, in \cref{lemma: finite,proposition: finite,remark: finite} we provide concrete criteria to determine which automorphisms induced by semigroup automorphisms have non-trivial image in the quotient $\operatorname{Aut}_\KK(B/I)/\operatorname{Aut}^{0}_\KK(B/I)$, using the generating set of the ideal $I$. In particular, these results allow us to decide when $\operatorname{Aut}_\KK(B/I)$ is connected.

A novel application of this result is a proof that these monomial algebras are determined by their automorphism groups. Indeed, the first author of this article proved in a recent preprint that if $\operatorname{Aut}^{0}(\mathbf{k}[x_1,\dots,x_n]/I) \cong \operatorname{Aut}^{0}(\mathbf{k}[x_1,\dots,x_m]/J)$ as algebraic groups, then $\mathbf{k}[x_1,\dots,x_n]/I \cong \mathbf{k}[x_1,\dots,x_m]/J$
\cite[Theorem 1.2]{DLA24}.

Let now $B$ be an affine domain, and consider $ \operatorname{Aut}_\KK(B) $ with its structure of ind-group \cite{FK18}. Assume that every automorphism of the connected component of the identity $\Aut_\KK^0(B)$ of $\operatorname{Aut}_\KK(B)$ is algebraic, that is, for every $g\in \operatorname{Aut}_\KK^0(B)$ there exists an algebraic subgroup $G$ of $\operatorname{Aut}_\KK^0(B)$ such that $g\in G$. Then, $\operatorname{Aut}_\KK^0(B)$ is a solvable group with derived length at most two (see \cite[Theorem 1.1]{PR24}; see also \cite[Theorem 1.3]{PR23}). In contrast, the automorphism group of a zero-dimensional monomial algebra $B$ does not necessarily have a solvable connected component (see \cref{rem:non-solv}). We characterize when this is the case in \cref{cor:solvable}.

\medskip

The article is organized as follows: In \cref{section 1}, we present preliminaries on semigroup algebras, monomial ideals, and homogeneous derivations. \Cref{section 2} explores the relationship between derivations on a semigroup algebra and its quotient monomial algebras. In \cref{section 3}, we classify liftable locally nilpotent derivations on monomial algebras. Finally, in \cref{section 4}, we describe the automorphism group of zero-dimensional monomial algebras that are quotients of the polynomial ring.

\subsection*{Acknowledgments}

We are deeply grateful to the anonymous referee of an earlier version of this manuscript for their thorough and constructive review, which greatly helped us to improve the quality of the paper. This collaboration began at the AGREGA2 conference, which took place at Universidad Técnica Federico Santa María in March 2023. We extend our gratitude to the institution for its support and hospitality. We also thank Giancarlo Lucchini-Arteche for fruitful discussions regarding the topic.

\section{Preliminaries}\label{section 1}

In this section, we gather some preliminary results that are essential for the subsequent sections.

\subsection{Semigroup algebras, monomial ideals and homogeneous derivations}\label{Section:preliminary}

Throughout this article, we fix a free abelian group $M$ of rank $n$, that is, $M \cong \ZZ^n$. We also consider $N$ as the dual group, defined by $N = \operatorname{Hom}(M, \ZZ)$. There exists a natural duality pairing 
$\langle\ ,\ \rangle \colon M \times N \to \ZZ$ given by $\langle m, p \rangle = p(m)
$. Let $L$ be any field. We define the $L$-vector spaces $M_L = M \otimes_\ZZ L$ and $N_L = N \otimes_\ZZ L$. The aforementioned duality pairing extends naturally to a pairing 
$\langle\ ,\ \rangle \colon M_L \times N_L \to L$.

An affine semigroup $S$ is a finitely generated semigroup with an identity element that can be embedded into a lattice. From now on, we fix a minimal embedding $S\hookrightarrow M$ in the sense that $ M $ does not have a sublattice containing $ S $. The cone $\omega_S$ of $S$ is the cone spanned by $S$ inside $M_\RR$. We say that an affine semigroup $S$ is saturated if $\omega_S \cap M = S$, and pointed if the only vector space contained in $\omega_S$ is $\{0\}$. Unless otherwise stated, all semigroups in the sequel are assumed to be affine, saturated, pointed, and minimally embedded.

Let $S$ be a semigroup. We say that a subsemigroup $S' \subset S$ is a face of $S$ if for all $m, m' \in S$ with $m + m' \in S'$, we have $m, m' \in S'$. Note that $S' \subset S$ is a face if and only if $S' = S \cap \tau$, where $\tau$ is a polyhedral face of $\omega_S$. A face of $S$ isomorphic to $\ZZ_{\geq0}$ as a semigroup is called a ray, and a maximal proper face of $S$ is called a facet. The rays of $S$ are given by $S \cap \tau$ where $\tau$ is a one-dimensional polyhedral face of $\omega_S$, and the facets of $S$ are given by $S \cap \tau$ where $\tau$ is a $(n-1)$-dimensional polyhedral face of $\omega_S$. A ray of $S$ is uniquely determined by its unique generator as a semigroup. We denote by $S(1) \subset S$ the set of all these generators of all the rays. The rank of a face $F \subset S$, denoted by $\rank F$, is the rank of the subgroup of $M$ generated by $F$. A semigroup is called simplicial if $S(1)$ is an $\RR$-basis of $M_\RR$.

For every semigroup $S$, we define its dual semigroup as
$$S^\vee=\{p\in N\mid p(m)\geq 0,\mbox{ for all } m\in S\}\,.$$
We observe that $S^\vee$ is an affine saturated semigroup, see \cite[Proposition~1.2.4]{cox2011toric}. Given a face $F \subseteq S$, its dual face \cite[Proposition~1.2.10 (a)]{cox2011toric} is defined as $F^* := F^\perp \cap S^\vee$ ,
where $F^\perp := \{ p \in N \mid p(m) = 0 \text{ for all } m \in F \}$.
This construction yields a face $F^* \subseteq S^\vee$, and defines an inclusion-reversing bijection between the set of faces of $S$ and the set of faces of $S^\vee$ \cite[Proposition~1.2.10 (b)]{cox2011toric}. In particular, for any face $F \subseteq S$, the double dual satisfies $(F^*)^* = F$. Moreover, $\operatorname{rank}(F) + \operatorname{rank}(F^*) = n$; see \cite[Proposition~1.2.10 (c)]{cox2011toric}.

The semigroup algebra $\KK[S]$ of $S$ is defined as:
$$\KK[S]=\bigoplus_{m\in S}\KK\cdot\XX^m\quad\mbox{where}\quad \XX^m\cdot\XX^{m'}=\XX^{m+m'}\quad\mbox{and}\quad \XX^0=1\,.$$
The symbols $\XX^m$, with $m\in M$ are called monomials. Note that since $S$ is affine and saturated, $\KK[S]$ is a normal affine domain. 

An ideal $I\subset \KK [S]$ is called monomial if there exists $l\in\ZZ_{>0}$ and $\mathbf{a}_{1},\ldots,\mathbf{a}_{l}\in S$ such that $I=(\XX^{\mathbf{a}_1},\dots,\XX^{\mathbf{a}_l})$. The support of a monomial ideal $I\subset \KK[S]$ is denoted by $\supp(I),$ and is given by $$\supp(I)=\{m\in S\mid \mathbf{x}^m\in I\}=\bigcup_{k=1}^l (\mathbf{a}_k+S)\,.$$
Let $I\subset \KK[S]$ be a monomial ideal. We say that $I$ has cofinite support if $S\setminus \supp(I)$ is finite. 

\begin{example}[The polynomial ring] \label{ex:polynomial-ring} 
Let $M=\mathbb{Z}^{n}$ and let $E=\{e_{1},\ldots,e_{n}\}$ be its canonical basis. Let $S=\ZZ_{\geq 0}^n.$ There exists a natural isomorphism from $\KK[S]$ to 
 the polynomial ring $\KK[x_1,\ldots,x_n]$, given by
$\XX^{e_i}\mapsto x_i$, for $1\leq i\leq n$. Under this isomorphism, for every $m=(m_{1},\ldots,m_{n})\in S$, the monomial $\XX^{m}$ corresponds to $x_{1}^{m_{1}}\cdots x_{n}^{m_{n}}$.

Now, fixing $n=2$ we have that $\KK[S]=\KK[x,y]$. Let $I_{1}=(x^2y^5,x^3y^2,x^5)\subset\KK[S]$ and $I_2=(y^5,x^3y^2,x^6)\subset\KK[S]$. Note that $I_2$ is a monomial ideal with cofinite support while $I_1$ is not. The ideal $I_1$ is represented in \cref{fig:1-ex1} and $I_2$ is represented in \cref{fig:2-ex1}. 
The solid dots $\bullet$ represent elements in $\supp(I_i),$ while the hollow dots $\circ$ represent elements in $m\in S\setminus\supp(I_i)$, for $1\leq i\leq 2$.

\begin{figure}[H]
\begin{minipage}[c]{0.45\textwidth}
\centering
\begin{picture}(100,95)
\definecolor{gray1}{gray}{0.7}
\definecolor{gray2}{gray}{0.85}
\definecolor{green}{RGB}{0,124,0}

\textcolor{gray2}{\put(0,20){\vector(1,0){90}}}

\textcolor{gray2}{\put(10,10){\vector(0,1){88}}}

\put(0,10){\textcolor{gray1}{\circle*{3}}}
\put(10,10){\textcolor{gray1}{\circle*{3}}}
\put(20,10){\textcolor{gray1}{\circle*{3}}} 
\put(30,10){\textcolor{gray1}{\circle*{3}}} 
\put(40,10){\textcolor{gray1}{\circle*{3}}} 
\put(50,10){\textcolor{gray1}{\circle*{3}}} 
\put(60,10){\textcolor{gray1}{\circle*{3}}} 
\put(70,10){\textcolor{gray1}{\circle*{3}}} 
\put(80,10){\textcolor{gray1}{\circle*{3}}}

\put(0,20){\textcolor{gray1}{\circle*{3}}}
\put(10,20){\circle{3}} 
\put(20,20){\circle{3}} 
\put(30,20){\circle{3}} 
\put(40,20){\circle{3}}
\put(50,20){\circle{3}}
\put(60,20){\circle*{3}}
\put(70,20){\circle*{3}}
\put(80,20){\circle*{3}}

\put(0,30){\textcolor{gray1}{\circle*{3}}}
\put(10,30){{\circle{3}}}
\put(20,30){\circle{3}} 
\put(30,30){\circle{3}} 
\put(40,30){\circle{3}}
\put(50,30){\circle{3}}
\put(60,30){{\circle*{3}}} 
\put(70,30){\circle*{3}}
\put(80,30){\circle*{3}}

\put(0,40){\textcolor{gray1}{\circle*{3}}}
\put(10,40){{\circle{3}}}
\put(20,40){{\circle{3}}}
\put(30,40){\circle{3}} 
\put(40,40){{\circle*{3}}}
\put(50,40){{\circle*{3}}}
\put(60,40){\circle*{3}} 
\put(70,40){\circle*{3}}
\put(80,40){\circle*{3}}

\put(0,50){\textcolor{gray1}{\circle*{3}}}
\put(10,50){{\circle{3}}}
\put(20,50){{\circle{3}}}
\put(30,50){{\circle{3}}}
\put(40,50){\circle*{3}}
\put(50,50){\circle*{3}}
\put(60,50){\circle*{3}} 
\put(70,50){\circle*{3}} 
\put(80,50){\circle*{3}}

\put(0,60){\textcolor{gray1}{\circle*{3}}}
\put(10,60){{\circle{3}}}
\put(20,60){{\circle{3}}}
\put(30,60){{\circle{3}}}
\put(40,60){\circle*{3}}
\put(50,60){\circle*{3}}
\put(60,60){\circle*{3}} 
\put(70,60){\circle*{3}}
\put(80,60){\circle*{3}}

\put(0,70){\textcolor{gray1}{\circle*{3}}}
\put(10,70){{\circle{3}}}
\put(20,70){{\circle{3}}}
\put(30,70){{\circle*{3}}}
\put(40,70){{\circle*{3}}}
\put(50,70){\circle*{3}}
\put(60,70){\circle*{3}} 
\put(70,70){\circle*{3}} 
\put(80,70){\circle*{3}}

\put(0,80){\textcolor{gray1}{\circle*{3}}}
\put(10,80){{\circle{3}}}
\put(20,80){{\circle{3}}}
\put(30,80){{\circle*{3}}}
\put(40,80){{\circle*{3}}}
\put(50,80){{\circle*{3}}}
\put(60,80){\circle*{3}} 
\put(70,80){\circle*{3}} 
\put(80,80){\circle*{3}}

\put(0,90){\textcolor{gray1}{\circle*{3}}}
\put(10,90){{\circle{3}}}
\put(20,90){{\circle{3}}}
\put(30,90){{\circle*{3}}}
\put(40,90){{\circle*{3}}}
\put(50,90){{\circle*{3}}}
\put(60,90){{\circle*{3}}}
\put(70,90){\circle*{3}} 
\put(80,90){\circle*{3}} 
\end{picture}
\centering
\caption{$I_1=(x^2y^5,x^3y^2,x^5)$}\label{fig:1-ex1}
\end{minipage}
 \hfill
\begin{minipage}[c]{0.45\textwidth}

\centering
\begin{picture}(100,95)
\definecolor{gray1}{gray}{0.7}
\definecolor{gray2}{gray}{0.85}
\definecolor{green}{RGB}{0,124,0}

\textcolor{gray2}{\put(0,20){\vector(1,0){90}}}

\textcolor{gray2}{\put(10,10){\vector(0,1){88}}}

\put(0,10){\textcolor{gray1}{\circle*{3}}}
\put(10,10){\textcolor{gray1}{\circle*{3}}}
\put(20,10){\textcolor{gray1}{\circle*{3}}} 
\put(30,10){\textcolor{gray1}{\circle*{3}}} 
\put(40,10){\textcolor{gray1}{\circle*{3}}} 
\put(50,10){\textcolor{gray1}{\circle*{3}}} 
\put(60,10){\textcolor{gray1}{\circle*{3}}} 
\put(70,10){\textcolor{gray1}{\circle*{3}}} 
\put(80,10){\textcolor{gray1}{\circle*{3}}}

\put(0,20){\textcolor{gray1}{\circle*{3}}}
\put(10,20){\circle{3}} 
\put(20,20){\circle{3}} 
\put(30,20){\circle{3}} 
\put(40,20){\circle{3}}
\put(50,20){\circle{3}}
\put(60,20){\circle{3}}
\put(70,20){\circle*{3}}
\put(80,20){\circle*{3}}

\put(0,30){\textcolor{gray1}{\circle*{3}}}
\put(10,30){{\circle{3}}}
\put(20,30){\circle{3}} 
\put(30,30){\circle{3}} 
\put(40,30){\circle{3}}
\put(50,30){\circle{3}}
\put(60,30){{\circle{3}}} 
\put(70,30){\circle*{3}}
\put(80,30){\circle*{3}}

\put(0,40){\textcolor{gray1}{\circle*{3}}}
\put(10,40){{\circle{3}}}
\put(20,40){{\circle{3}}}
\put(30,40){\circle{3}} 
\put(40,40){{\circle*{3}}}
\put(50,40){{\circle*{3}}}
\put(60,40){\circle*{3}} 
\put(70,40){\circle*{3}}
\put(80,40){\circle*{3}}

\put(0,50){\textcolor{gray1}{\circle*{3}}}
\put(10,50){{\circle{3}}}
\put(20,50){{\circle{3}}}
\put(30,50){{\circle{3}}}
\put(40,50){\circle*{3}}
\put(50,50){\circle*{3}}
\put(60,50){\circle*{3}} 
\put(70,50){\circle*{3}} 
\put(80,50){\circle*{3}}

\put(0,60){\textcolor{gray1}{\circle*{3}}}
\put(10,60){{\circle{3}}}
\put(20,60){{\circle{3}}}
\put(30,60){{\circle{3}}}
\put(40,60){\circle*{3}}
\put(50,60){\circle*{3}}
\put(60,60){\circle*{3}} 
\put(70,60){\circle*{3}}
\put(80,60){\circle*{3}}

\put(0,70){\textcolor{gray1}{\circle*{3}}}
\put(10,70){{\circle*{3}}}
\put(20,70){{\circle*{3}}}
\put(30,70){{\circle*{3}}}
\put(40,70){{\circle*{3}}}
\put(50,70){\circle*{3}}
\put(60,70){\circle*{3}} 
\put(70,70){\circle*{3}} 
\put(80,70){\circle*{3}}

\put(0,80){\textcolor{gray1}{\circle*{3}}}
\put(10,80){{\circle*{3}}}
\put(20,80){{\circle*{3}}}
\put(30,80){{\circle*{3}}}
\put(40,80){{\circle*{3}}}
\put(50,80){{\circle*{3}}}
\put(60,80){\circle*{3}} 
\put(70,80){\circle*{3}} 
\put(80,80){\circle*{3}}

\put(0,90){\textcolor{gray1}{\circle*{3}}}
\put(10,90){{\circle*{3}}}
\put(20,90){{\circle*{3}}}
\put(30,90){{\circle*{3}}}
\put(40,90){{\circle*{3}}}
\put(50,90){{\circle*{3}}}
\put(60,90){{\circle*{3}}}
\put(70,90){\circle*{3}} 
\put(80,90){\circle*{3}} 
\end{picture}
\centering

\caption{$I_2=(y^5,x^3y^2,x^6)$}\label{fig:2-ex1}
\end{minipage}

\end{figure}

\end{example}

The algebra $\KK[S]$ is naturally $M$-graded, and under this grading every monomial ideal $I$ is a graded ideal. Thus, the quotient $\KK[S]/I$ inherits an $M$-grading given by $\deg \quot{\XX}^m=m$ if $m\notin \supp(I)$, where $\quot{\XX}^m$ denotes the image of $\XX^m$ within $\KK[S]/I$. Recall that if $m\in \supp(I)$, then $\quot{\XX}^m=0$. In the sequel, we will consider $\KK[S]/I$ as an $M$-graded algebra. 

\medskip

Let $B$ be a $\KK$-algebra. A derivation $\partial$ of $B$ is a $\KK$-linear map satisfying the Leibniz rule, that is, 
$$
\partial(fg)=\partial(f)g+f\partial(g), \quad \text{for all } f,g \in B.
$$
We denote by $\Der(B)$ the $\KK$-vector space of derivations on $B$. Let $I\subset B$ be an ideal. We denote by $\Der_I(B)$ the subspace of derivations $\partial \in \Der(B)$ such that $\partial(I)\subset I$. A derivation $\partial$ is locally nilpotent if for each $f\in B$ there exists $\ell\in \ZZ_{>0}$ such that $\partial^\ell(f)=0$, where $\partial^\ell$ denotes $\partial$ composed with itself $\ell$ times. Note that if $B$ is finite-dimensional, then every locally nilpotent derivation is nilpotent. In this case, we will simply use the term nilpotent.

Letting $S$ be a semigroup, we let now $\partial$ be a derivation on $\KK[S]$. We say that $\partial$ is homogeneous if it sends homogeneous elements into homogeneous elements with respect to the $M$-grading of $\KK[S]$. If $\partial$ is homogeneous, by \cite[Lemma~1.1]{DL24}, there exists a unique element $\alpha\in M$, called the degree of $\partial$ and denoted by $\deg\partial$, such that for every $\XX^m\notin\ker\partial$ we have $\partial(\XX^m)=\lambda\XX^{m+\alpha}$ for some $\lambda\in \KK$. Furthermore, by \cite[Proposition 2.1]{DL24}, every derivation in $\Der(\KK[S])$ admits a finite decomposition $\partial=\sum\partial_{\alpha}$, where each $\partial_{\alpha}$ is a homogeneous derivation on $\KK[S]$ of degree $\alpha\in M$. 

\begin{definition}
We say that a homogeneous derivation $\partial$ is inner if $\deg\partial\in S$ and outer if $\deg\partial\in M\setminus S$. 
\end{definition}

A straightforward application of the Leibniz rule shows that every homogeneous derivation on $\KK[S]$ is of the form
\begin{align}\label{equation:homogeneous}
\partial_{\alpha,p}\colon \KK[S]\to \KK[S],\quad\XX^m\mapsto p(m)\cdot\XX^{m+\alpha}\,,
\end{align}
 where $\alpha=\deg\partial_{\alpha,p}\in M$ and $p\in N_\KK=N\otimes_\ZZ\KK$. Conversely, every such map defines a derivation of the ambient algebra $\KK[M]$, but does not necessarily restrict to a derivation of $\KK[S]$.

Let now $S$ be an affine semigroup and $I\subset \KK[S]$ be a monomial ideal. The Lie algebras $\operatorname{Der}(\KK [S])$ and $\operatorname{Der}_I(\KK [S])$ have been studied by several authors, see for instance \cite{B95,T09,liendo2010affine,KLL15,DL24}. We present now the description of derivations on $\KK[S]$ given in \cite[Proposition~3.1]{KLL15} with the notation adapted to our context. To describe inner derivations, we now let
$$\mathfrak{g}=\bigoplus_{\alpha\in S} \mathfrak{g}_\alpha\subset \Der(\KK[S])\quad\mbox{where}\quad \mathfrak{g}_\alpha=\left\{\partial_{\alpha,p}\mid p\in N_\KK\right\}.$$
For every $\alpha\in S$, we have that $\mathfrak{g}_{\alpha}$ is a vector space isomorphic to $N_\KK,$ since $\partial_{\alpha,p}+\partial_{\alpha,q}=\partial_{\alpha,p+q},$ for all $p,q\in N_\KK$.

To describe outer derivations, which correspond exactly to the homogeneous locally nilpotent derivations in the notation of \cite[Proposition~3.1]{KLL15}, we need to introduce some notation. Recall that $S^\vee$ denotes the dual semigroup of $S$. For every $\rho\in S^\vee(1)$ we define the set
\begin{align}\label{eq:R}
\mathcal{R}=\bigcup_{\rho\in S^\vee(1)}\mathcal{R}_\rho\quad \mbox{where}\quad \mathcal{R}_\rho=\left\{\alpha\in M\mid \rho(\alpha)=-1\mbox{ and } \rho'(\alpha)\geq 0\mbox{ for all } \rho'\in S^\vee(1)\setminus\{\rho\}\right\}\,.
\end{align}
Elements in $\mathcal{R}$ are called the Demazure roots of the semigroup $S$. Note that by definition, $\mathcal{R}_\rho$ is disjoint from $\mathcal{R}_{\rho'}$ whenever $\rho\neq \rho'$. Indeed, for every $\alpha \in \mathcal{R}_\rho$, we have that $\rho'(\alpha)\geq 0$ and so $\alpha\notin \mathcal{R}_{\rho'}$. 
We now let $$\mathfrak{s}=\bigoplus_{\alpha\in \mathcal{R}}\mathfrak{s}_\alpha\quad \mbox{where}\quad \mathfrak{s}_\alpha=\left\{\partial_{\alpha,p}\mid p=\lambda\cdot \rho, \mbox{ with } \lambda\in \KK \mbox{ and } \rho\in S^\vee(1) \mbox{ such that } \alpha\in\mathcal{R}_\rho\right\}$$
is the $1$-dimensional vector space of outer derivations on degree $\alpha$.
\begin{proposition}[{\cite[Proposition~3.1]{KLL15}}] \label{decomposition-no-I}
Let $S$ be a saturated affine semigroup. Then
$$\Der(\KK[S])=\mathfrak{g}\oplus \mathfrak{s}\,,$$ where homogeneous derivations in $\mathfrak{g}$ are inner and homogeneous derivations in $\mathfrak{s}$ are outer.
\end{proposition}

We now provide a description of $\Der_I(\KK[S])$. The image $\partial(I)$ is contained in $I$ for all derivations $\partial\in \mathfrak{g}$. So $\mathfrak{g}\subset \Der_I(\KK[S])$. For outer derivations, let $\mathfrak{s}(I)=\{\partial\in \mathfrak{s}\mid \partial(I)\subset I\}$. Letting $I\subset \KK[S]$ be the ideal $I=(\XX^{\mathbf{a}_1},\dots,\XX^{\mathbf{a}_l})$ with $\mathbf{a}_i\in S$, we let 

\begin{equation}\label{eq:R_I}
    \mathcal{R}(I)=\bigcup_{\rho\in S^\vee(1)} \mathcal{R}_\rho(I) \quad\mbox{where} \quad
    \mathcal{R}_\rho(I)=\left\{\alpha\in \mathcal{R}_\rho\mid \mathbf{a}_i+\alpha\in \supp(I) \mbox{ for all } \mathbf{a}_i\notin \rho^\bot\right\}\,.
\end{equation}
The following proposition provides a description of $\Der_I(\KK[S])$.

\begin{proposition}\label{decomposition-derivations} 
Let $S$ be a saturated affine semigroup and let $I=(\XX^{\mathbf{a}_1},\dots,\XX^{\mathbf{a}_l})$ with $\mathbf{a}_i\in S$. Then
 $$\mathfrak{s}(I)=\bigoplus_{\alpha\in \mathcal{R}(I)}\mathfrak{s}_{\alpha}.$$ In particular $\Der_I(\KK[S])=\mathfrak{g}\oplus\mathfrak{s}(I).$
\end{proposition}
 \begin{proof} Let $\partial\in \mathfrak{s}(I).$ Then there exists $r\in\mathbb{Z}_{> 0}$ such that $\partial=\sum_{i=1}^r \partial_{\alpha_i},$ where $\alpha_{i}=\deg \partial_{\alpha_{i}},\:\partial_{\alpha_{i}}\in\mathfrak{s}_{\alpha_{i}}$ and $\partial_{\alpha_i}$ is homogeneous for all $i\in\{1,\ldots,r\}.$ Let $\alpha=\alpha_{i},$ for some $i\in\{1,\ldots,r\},$ and let $\rho\in N$ be such that $\alpha\in\mathcal{R}_{\rho}.$ We claim that $\alpha\in \mathcal{R}_{\rho}(I).$ Let  $\mathbf{a}\in\{\mathbf{a}_{1},\ldots,\mathbf{a}_{l}\}\setminus\rho^{\perp}.$ Then $\partial(\XX^{\mathbf{a}})=\sum\partial_{\alpha_{i}}(\XX^{\mathbf{a}})\in I.$ It follows by \cite[Corollary 1.1.3]{HH11} that $\partial_{\alpha}(\XX^{\mathbf{a}})\in I,$ that is $\mathbf{a}+\alpha\in\supp(I).$ Hence $\alpha\in\mathcal{R}_{\rho}(I).$ The second statement follows directly from \Cref{decomposition-no-I}.
 \end{proof}

 The above description of $\mathfrak{s}(I)$ was given in \cite[Theorem~2.2]{T09} for the case where $\KK[S]$ is the polynomial ring, as in \cref{ex:polynomial-ring}, see also \cite[Theorem~2.2.1]{B95}. We apply \cref{decomposition-derivations} to the polynomial ring in the following example.

 \begin{example}[derivations on the polynomial ring]\label{example: homogeneous derivations on the polynomial ring}
With the notation of \cref{ex:polynomial-ring}, we have that the inner derivations on the polynomial ring
are of the form
$$\partial_{\alpha,p}=x_1^{\alpha_1}\ldots x_n^{\alpha_n}\cdot\left(p_1x_1\frac{d}{dx_1}+\cdots+p_nx_n\frac{d}{dx_n}\right),$$
for $\alpha=(\alpha_1,\ldots,\alpha_n)\in \ZZ_{\geq0}^n$ and $p=(p_1,\dots,p_n) \in \KK^n.$ We now compute the outer derivations. Let $E^{*}=\{e_{1}^{*},\ldots,e_{n}^{*}\}$ be the canonical basis of $N.$ Then $E^{*}$ is the set of generating vectors of the rays of $S^{*}.$ Thus, for $1\leq j\leq n,$ we have $$\mathcal{R}_{e_j^*}=\big\{\alpha=(\alpha_1,\ldots,\alpha_n)\in \ZZ^n \mid \alpha_j=-1 \mbox{ and } \alpha_i\geq 0 \mbox{ for all } i\neq j\big\}.$$
Hence, $\mathcal{R}=\bigcup_{j=1}^{n}\mathcal{R}_{e_{j}^{*}}$ contains the degree of all the outer derivations on $\KK[S].$ Thus, for $1\leq j\leq n,$ the derivation $\partial_{\alpha,p},$ where $p=\lambda\cdot e_j^*,$ for some $\lambda\in\KK^{*}$ and $\alpha=(\alpha_{1},\ldots,\alpha_{n})\in \mathcal{R}_{e_{j}^{*}},$ corresponds to
$$
\partial_{\alpha,p}=\lambda\cdot x_1^{\alpha_1}\cdots \widehat{x_j^{\alpha_j}}\cdots x_n^{\alpha_n}\cdot\frac{d}{dx_j}\,,
$$
where $\widehat{x_j^{\alpha_j}}$ means that the factor is excluded from the product.
\end{example}

\begin{example} With the notation of \cref{ex:polynomial-ring}, we let $n=2$. Let $I=I_1\subset\KK[x,y]$ be the monomial ideal $(x^2y^5,x^3y^2,x^5)$. We compute $\mathcal{R}(I).$ The support $\supp(I)$ is presented in \cref{fig:1-ex2}. Let $\mathbf{a}_1=(2,5),\mathbf{a}_2=(3,2),\mathbf{a}_3=(5,0)\in\ZZ_{\geq 0}^{2}$ be the exponents of the generators of $I.$ We first note that $\mathcal{R}_{e_{1}^{*}}(I)$
is empty, because $\mathbf{a}_{1}+\alpha=(1,5+\ell)\notin\supp(I),$ for any $\alpha=(-1,\ell),$ where $\ell\in\ZZ_{\geq 0}.$ 

We now describe $\mathcal{R}_{e_{2}^{*}}(I).$ Note that $\mathcal{R}_{e_{2}^{*}}(I)$ is the set of elements $\alpha=(\ell,-1),$ where $\ell\in\ZZ_{\geq 0},$ such that $\alpha+\mathbf{a}_1=(2+\ell,4),\alpha+\mathbf{a}_2=(3+\ell,1)\in\supp(I),$ because $e_{2}^{*}(\mathbf{a}_{3})=0.$ Thus $\mathcal{R}_{e_{2}^{*}}(I)=\{(\ell,-1)\in \ZZ^{2}\mid \ell\geq 2\}.$ Hence, we have
$$\mathfrak{s}(I)=\bigoplus_{\alpha\in\mathcal{R}_{e_{2}^{*}}(I)}\mathfrak{s}_\alpha=\bigoplus_{\ell=2}^\infty \KK\cdot x^\ell\frac{d}{dy}=\left(\bigoplus_{\ell=2}^\infty \KK\cdot x^\ell\right)\frac{d}{dy}=x^2\KK[x]\cdot\frac{d}{dy}.$$
The set $\mathcal{R}(I)=\mathcal{R}_{e_{2}^{*}}(I)$ is given by the red dots in \cref{fig:2-ex2} represents the degrees of outer derivations $\partial\in\Der_{I}(\KK[x,y])$. Finally, we observe that all red dots represent zero derivations on $\KK[x,y]/I,$ except those of degree $\alpha\in \{(2,-1),(3,-1),(4,-1)\}$.

\begin{figure}[H]
\begin{minipage}[c]{0.45\textwidth}
\centering
\begin{picture}(100,95)
\definecolor{gray1}{gray}{0.7}
\definecolor{gray2}{gray}{0.85}
\definecolor{green}{RGB}{0,124,0}

\textcolor{gray2}{\put(0,20){\vector(1,0){90}}}

\textcolor{gray2}{\put(10,10){\vector(0,1){88}}}

\put(0,10){\textcolor{gray1}{\circle*{3}}}
\put(10,10){\textcolor{gray1}{\circle*{3}}}
\put(20,10){\textcolor{gray1}{\circle*{3}}} 
\put(30,10){\textcolor{gray1}{\circle*{3}}} 
\put(40,10){\textcolor{gray1}{\circle*{3}}} 
\put(50,10){\textcolor{gray1}{\circle*{3}}} 
\put(60,10){\textcolor{gray1}{\circle*{3}}} 
\put(70,10){\textcolor{gray1}{\circle*{3}}} 
\put(80,10){\textcolor{gray1}{\circle*{3}}}

\put(0,20){\textcolor{gray1}{\circle*{3}}}
\put(10,20){\circle{3}} 
\put(20,20){\circle{3}} 
\put(30,20){\circle{3}} 
\put(40,20){\circle{3}}
\put(50,20){\circle{3}}
\put(60,20){\circle*{3}}
\put(70,20){\circle*{3}}
\put(80,20){\circle*{3}}

\put(0,30){\textcolor{gray1}{\circle*{3}}}
\put(10,30){{\circle{3}}}
\put(20,30){\circle{3}} 
\put(30,30){\circle{3}} 
\put(40,30){\circle{3}}
\put(50,30){\circle{3}}
\put(60,30){{\circle*{3}}} 
\put(70,30){\circle*{3}}
\put(80,30){\circle*{3}}

\put(0,40){\textcolor{gray1}{\circle*{3}}}
\put(10,40){{\circle{3}}}
\put(20,40){{\circle{3}}}
\put(30,40){\circle{3}} 
\put(40,40){{\circle*{3}}}
\put(50,40){{\circle*{3}}}
\put(60,40){\circle*{3}} 
\put(70,40){\circle*{3}}
\put(80,40){\circle*{3}}

\put(0,50){\textcolor{gray1}{\circle*{3}}}
\put(10,50){{\circle{3}}}
\put(20,50){{\circle{3}}}
\put(30,50){{\circle{3}}}
\put(40,50){\circle*{3}}
\put(50,50){\circle*{3}}
\put(60,50){\circle*{3}} 
\put(70,50){\circle*{3}} 
\put(80,50){\circle*{3}}

\put(0,60){\textcolor{gray1}{\circle*{3}}}
\put(10,60){{\circle{3}}}
\put(20,60){{\circle{3}}}
\put(30,60){{\circle{3}}}
\put(40,60){\circle*{3}}
\put(50,60){\circle*{3}}
\put(60,60){\circle*{3}} 
\put(70,60){\circle*{3}}
\put(80,60){\circle*{3}}

\put(0,70){\textcolor{gray1}{\circle*{3}}}
\put(10,70){{\circle{3}}}
\put(20,70){{\circle{3}}}
\put(30,70){{\circle*{3}}}
\put(40,70){{\circle*{3}}}
\put(50,70){\circle*{3}}
\put(60,70){\circle*{3}} 
\put(70,70){\circle*{3}} 
\put(80,70){\circle*{3}}

\put(0,80){\textcolor{gray1}{\circle*{3}}}
\put(10,80){{\circle{3}}}
\put(20,80){{\circle{3}}}
\put(30,80){{\circle*{3}}}
\put(40,80){{\circle*{3}}}
\put(50,80){{\circle*{3}}}
\put(60,80){\circle*{3}} 
\put(70,80){\circle*{3}} 
\put(80,80){\circle*{3}}

\put(0,90){\textcolor{gray1}{\circle*{3}}}
\put(10,90){{\circle{3}}}
\put(20,90){{\circle{3}}}
\put(30,90){{\circle*{3}}}
\put(40,90){{\circle*{3}}}
\put(50,90){{\circle*{3}}}
\put(60,90){{\circle*{3}}}
\put(70,90){\circle*{3}} 
\put(80,90){\circle*{3}} 
\end{picture}
\centering
\caption{$I=(x^2y^5,x^3y^2,x^5)$}\label{fig:1-ex2}

\end{minipage}
 \hfill
\begin{minipage}[c]{0.48\textwidth}
\centering
\begin{picture}(100,90)
\definecolor{gray1}{gray}{0.7}
\definecolor{gray2}{gray}{0.85}
\definecolor{green}{RGB}{0,124,0}

\textcolor{gray2}{\put(0,20){\vector(1,0){90}}}

\textcolor{gray2}{\put(10,10){\vector(0,1){88}}}

\put(0,10){\textcolor{gray1}{\circle*{3}}}
\put(10,10){\textcolor{gray1}{\circle*{3}}}
\put(20,10){\textcolor{gray1}{\circle*{3}}} 
\put(30,10){\textcolor{red}{\circle*{3}}} 
\put(40,10){\textcolor{red}{\circle*{3}}} 
\put(50,10){\textcolor{red}{\circle*{3}}} 
\put(60,10){\textcolor{red}{\circle*{3}}} 
\put(70,10){\textcolor{red}{\circle*{3}}} 
\put(80,10){\textcolor{red}{\circle*{3}}}

\put(0,20){\textcolor{gray1}{\circle*{3}}}
\put(10,20){\circle{3}} 
\put(20,20){\circle{3}} 
\put(30,20){\circle{3}} 
\put(40,20){\circle{3}}
\put(50,20){\circle{3}}
\put(60,20){\circle*{3}}
\put(70,20){\circle*{3}}
\put(80,20){\circle*{3}}

\put(0,30){\textcolor{gray1}{\circle*{3}}}
\put(10,30){{\circle{3}}}
\put(20,30){\circle{3}} 
\put(30,30){\circle{3}} 
\put(40,30){\circle{3}}
\put(50,30){\circle{3}}
\put(60,30){{\circle*{3}}} 
\put(70,30){\circle*{3}}
\put(80,30){\circle*{3}}

\put(0,40){\textcolor{gray1}{\circle*{3}}}
\put(10,40){{\circle{3}}}
\put(20,40){{\circle{3}}}
\put(30,40){\circle{3}} 
\put(40,40){{\circle*{3}}}
\put(50,40){{\circle*{3}}}
\put(60,40){\circle*{3}} 
\put(70,40){\circle*{3}}
\put(80,40){\circle*{3}}

\put(0,50){\textcolor{gray1}{\circle*{3}}}
\put(10,50){{\circle{3}}}
\put(20,50){{\circle{3}}}
\put(30,50){{\circle{3}}}
\put(40,50){\circle*{3}}
\put(50,50){\circle*{3}}
\put(60,50){\circle*{3}} 
\put(70,50){\circle*{3}} 
\put(80,50){\circle*{3}}

\put(0,60){\textcolor{gray1}{\circle*{3}}}
\put(10,60){{\circle{3}}}
\put(20,60){{\circle{3}}}
\put(30,60){{\circle{3}}}
\put(40,60){\circle*{3}}
\put(50,60){\circle*{3}}
\put(60,60){\circle*{3}} 
\put(70,60){\circle*{3}}
\put(80,60){\circle*{3}}

\put(0,70){\textcolor{gray1}{\circle*{3}}}
\put(10,70){{\circle{3}}}
\put(20,70){{\circle{3}}}
\put(30,70){{\circle*{3}}}
\put(40,70){{\circle*{3}}}
\put(50,70){\circle*{3}}
\put(60,70){\circle*{3}} 
\put(70,70){\circle*{3}} 
\put(80,70){\circle*{3}}

\put(0,80){\textcolor{gray1}{\circle*{3}}}
\put(10,80){{\circle{3}}}
\put(20,80){{\circle{3}}}
\put(30,80){{\circle*{3}}}
\put(40,80){{\circle*{3}}}
\put(50,80){{\circle*{3}}}
\put(60,80){\circle*{3}} 
\put(70,80){\circle*{3}} 
\put(80,80){\circle*{3}}

\put(0,90){\textcolor{gray1}{\circle*{3}}}
\put(10,90){{\circle{3}}}
\put(20,90){{\circle{3}}}
\put(30,90){{\circle*{3}}}
\put(40,90){{\circle*{3}}}
\put(50,90){{\circle*{3}}}
\put(60,90){{\circle*{3}}}
\put(70,90){\circle*{3}} 
\put(80,90){\circle*{3}}

\end{picture}
\centering

\caption{$\mathcal{R}(I)=\{(\ell,-1)\mid \ell \geq 2\}$}\label{fig:2-ex2}
\end{minipage}

\end{figure}

\end{example}

\section{Derivations and Monomial Ideals}\label{section 2}

Let $S$ be an affine semigroup and let $I$ be a monomial ideal of $\KK[S].$ Consider the natural projection map 
\begin{align} \label{projection}
 \pi\colon \Der_I(\KK[S]) \to \Der(\KK[S]/I),\quad \partial \mapsto \overline{\partial} \quad \text{defined by} \quad \overline{\partial}(\overline{a}) = \overline{\partial(a)}.
\end{align}
This map $\pi$ is not always surjective. Indeed, Kraft \cite[Example~5]{Kr2017vectorfields} demonstrates non-surjectivity in an example  where $S$ is non-saturated and $\KK[S]$ is isomorphic to the coordinate ring of the affine variety in $\AF^3$ defined by the equation $x^2z - y^2 = 0$, commonly known as Whitney's Umbrella. Below, we provide an example better suited to our context where $S$ is saturated.

\begin{example}
 Let $M = \ZZ^2$ and let $S$ be the semigroup generated by $\mu_1 = (1,0)$, $\mu_2 = (1,1)$, and $\mu_3=(1,2)$. Furthermore, $X = \spec \KK[S]$ is isomorphic to the variety in the affine space $\AF^3$ defined by the equation $xz - y^2 = 0$, by setting $x = \XX^{\mu_1}$, $y = \XX^{\mu_2}$, and $z = \XX^{\mu_3}.$ It follows from \cite[Lemma~2.4]{liendo2010affine} that $\mathcal{R}(S) = \{(\ell, -1) \in \ZZ^2 \mid \ell \in \ZZ_{\geq 0}\}\cup\{(0,1)+\ell(1,2)\in \ZZ^2\mid\ell\in\ZZ_{\geq 0}\},$ represented by the green points in \cref{fig:1-ex3-}. Let $I \subset \KK[S]$ be the monomial ideal generated by $\{\mathbf{x}^{\mathbf{a}_1},\dots,\mathbf{x}^{\mathbf{a}_5}\},$ where $\mathbf{a}_1 = (2,0)$, $\mathbf{a}_2 = (2,1)$, $\mathbf{a}_3 = (2,2)$, $\mathbf{a}_4 = (2,3)$ and $\mathbf{a}_5 = (2,4)$, that is $I = (x^2,x^2y,x^2y^2,x^2y^3,x^2y^4),$ shown in \cref{fig:2-ex3-}. We define the homogeneous derivation $\partial\colon \mathbf{k}[S]/I\to \mathbf{k}[S]/I,$ given by $\partial(\overline{x}\overline{y}^2)=\overline{x}$ and $\partial(\overline{x})=\partial(\overline{xy})=0.$ Note that $\partial$ is a homogeneous derivation on $\KK[S]/I$ of degree $(0,-2).$ Thus, $\pi$ cannot be surjective because $(0,-2) \notin \mathcal{R}(S)$.

\begin{figure}[H]
\begin{minipage}[c]{0.49\textwidth}
\centering
\begin{picture}(100,95)
\definecolor{gray1}{gray}{0.7}
\definecolor{gray2}{gray}{0.85}
\definecolor{green}{RGB}{0,124,0}

\textcolor{gray2}{\put(0,40){\vector(1,0){90}}}

\textcolor{gray2}{\put(10,10){\vector(0,1){88}}}

\put(0,10){\textcolor{gray1}{\circle*{3}}}
\put(10,10){\textcolor{gray1}{\circle*{3}}}
\put(20,10){\textcolor{gray1}{\circle*{3}}} 
\put(30,10){\textcolor{gray1}{\circle*{3}}} 
\put(40,10){\textcolor{gray1}{\circle*{3}}} 
\put(50,10){\textcolor{gray1}{\circle*{3}}} 
\put(60,10){\textcolor{gray1}{\circle*{3}}} 
\put(70,10){\textcolor{gray1}{\circle*{3}}} 
\put(80,10){\textcolor{gray1}{\circle*{3}}}

\put(0,20){\textcolor{gray1}{\circle*{3}}}
\put(10,20){\textcolor{gray1}{\circle*{3}}} 
\put(20,20){\textcolor{gray1}{\circle*{3}}} 
\put(30,20){\textcolor{gray1}{\circle*{3}}} 
\put(40,20){\textcolor{gray1}{\circle*{3}}}
\put(50,20){\textcolor{gray1}{\circle*{3}}}
\put(60,20){\textcolor{gray1}{\circle*{3}}}
\put(70,20){\textcolor{gray1}{\circle*{3}}}
\put(80,20){\textcolor{gray1}{\circle*{3}}}

\put(0,30){\textcolor{gray1}{\circle*{3}}}
\put(10,30){\textcolor{green}{\circle*{3}}}
\put(20,30){\textcolor{green}{\circle*{3}}}
\put(30,30){\textcolor{green}{\circle*{3}}}
\put(40,30){\textcolor{green}{\circle*{3}}}
\put(50,30){\textcolor{green}{\circle*{3}}}
\put(60,30){\textcolor{green}{\circle*{3}}} 
\put(70,30){\textcolor{green}{\circle*{3}}}
\put(80,30){\textcolor{green}{\circle*{3}}}

\put(0,40){\textcolor{gray1}{\circle*{3}}}
\put(10,40){{\circle{3}}}
\put(20,40){{\circle{3}}}
\put(30,40){\circle{3}} 
\put(40,40){{\circle{3}}}
\put(50,40){{\circle{3}}}
\put(60,40){\circle{3}} 
\put(70,40){\circle{3}}
\put(80,40){\circle{3}}

\put(0,50){\textcolor{gray1}{\circle*{3}}}
\put(10,50){\textcolor{green}{\circle*{3}}}
\put(20,50){{\circle{3}}}
\put(30,50){{\circle{3}}}
\put(40,50){\circle{3}}
\put(50,50){\circle{3}}
\put(60,50){\circle{3}} 
\put(70,50){\circle{3}} 
\put(80,50){\circle{3}}

\put(0,60){\textcolor{gray1}{\circle*{3}}}
\put(10,60){\textcolor{gray1}{\circle*{3}}}
\put(20,60){{\circle{3}}}
\put(30,60){{\circle{3}}}
\put(40,60){\circle{3}}
\put(50,60){\circle{3}}
\put(60,60){\circle{3}} 
\put(70,60){\circle{3}}
\put(80,60){\circle{3}}

\put(0,70){\textcolor{gray1}{\circle*{3}}}
\put(10,70){\textcolor{gray1}{\circle*{3}}}
\put(20,70){\textcolor{green}{\circle*{3}}}
\put(30,70){{\circle{3}}}
\put(40,70){{\circle{3}}}
\put(50,70){\circle{3}}
\put(60,70){\circle{3}} 
\put(70,70){\circle{3}} 
\put(80,70){\circle{3}}

\put(0,80){\textcolor{gray1}{\circle*{3}}}
\put(10,80){\textcolor{gray1}{\circle*{3}}}
\put(20,80){\textcolor{gray1}{\circle*{3}}}
\put(30,80){{\circle{3}}}
\put(40,80){{\circle{3}}}
\put(50,80){{\circle{3}}}
\put(60,80){\circle{3}} 
\put(70,80){\circle{3}} 
\put(80,80){\circle{3}}

\put(0,90){\textcolor{gray1}{\circle*{3}}}
\put(10,90){\textcolor{gray1}{\circle*{3}}}
\put(20,90){\textcolor{gray1}{\circle*{3}}}
\put(30,90){\textcolor{green}{\circle*{3}}}
\put(40,90){{\circle{3}}}
\put(50,90){{\circle{3}}}
\put(60,90){{\circle{3}}}
\put(70,90){\circle{3}} 
\put(80,90){\circle{3}}

\end{picture}
\centering
\caption{ The semigroup $S$}\label{fig:1-ex3-}
\end{minipage}
 \hfill
\begin{minipage}[c]{0.49\textwidth}
\centering
\begin{picture}(100,95)
\definecolor{gray1}{gray}{0.7}
\definecolor{gray2}{gray}{0.85}
\definecolor{green}{RGB}{0,124,0}

\textcolor{gray2}{\put(0,40){\vector(1,0){90}}}

\textcolor{gray2}{\put(10,10){\vector(0,1){88}}}

\put(0,10){\textcolor{gray1}{\circle*{3}}}
\put(10,10){\textcolor{gray1}{\circle*{3}}}
\put(20,10){\textcolor{gray1}{\circle*{3}}} 
\put(30,10){\textcolor{gray1}{\circle*{3}}} 
\put(40,10){\textcolor{gray1}{\circle*{3}}} 
\put(50,10){\textcolor{gray1}{\circle*{3}}} 
\put(60,10){\textcolor{gray1}{\circle*{3}}} 
\put(70,10){\textcolor{gray1}{\circle*{3}}} 
\put(80,10){\textcolor{gray1}{\circle*{3}}}

\put(0,20){\textcolor{gray1}{\circle*{3}}}
\put(10,20){\textcolor{green}{\circle*{3}}} 
\put(20,20){\textcolor{gray1}{\circle*{3}}} 
\put(30,20){\textcolor{gray1}{\circle*{3}}} 
\put(40,20){\textcolor{gray1}{\circle*{3}}}
\put(50,20){\textcolor{gray1}{\circle*{3}}}
\put(60,20){\textcolor{gray1}{\circle*{3}}}
\put(70,20){\textcolor{gray1}{\circle*{3}}}
\put(80,20){\textcolor{gray1}{\circle*{3}}}

\put(0,30){\textcolor{gray1}{\circle*{3}}}
\put(10,30){\textcolor{gray1}{\circle*{3}}}
\put(20,30){\textcolor{gray1}{\circle*{3}}}
\put(30,30){\textcolor{gray1}{\circle*{3}}}
\put(40,30){\textcolor{gray1}{\circle*{3}}}
\put(50,30){\textcolor{gray1}{\circle*{3}}}
\put(60,30){\textcolor{gray1}{\circle*{3}}} 
\put(70,30){\textcolor{gray1}{\circle*{3}}}
\put(80,30){\textcolor{gray1}{\circle*{3}}}

\put(0,40){\textcolor{gray1}{\circle*{3}}}
\put(10,40){{\circle{3}}}
\put(20,40){{\circle{3}}}
\put(30,40){\circle*{3}} 
\put(40,40){{\circle*{3}}}
\put(50,40){{\circle*{3}}}
\put(60,40){\circle*{3}} 
\put(70,40){\circle*{3}}
\put(80,40){\circle*{3}}

\put(0,50){\textcolor{gray1}{\circle*{3}}}
\put(10,50){\textcolor{gray1}{\circle*{3}}}
\put(20,50){{\circle{3}}}
\put(30,50){{\circle*{3}}}
\put(40,50){\circle*{3}}
\put(50,50){\circle*{3}}
\put(60,50){\circle*{3}} 
\put(70,50){\circle*{3}} 
\put(80,50){\circle*{3}}

\put(0,60){\textcolor{gray1}{\circle*{3}}}
\put(10,60){\textcolor{gray1}{\circle*{3}}}
\put(20,60){{\circle{3}}}
\put(30,60){{\circle*{3}}}
\put(40,60){\circle*{3}}
\put(50,60){\circle*{3}}
\put(60,60){\circle*{3}} 
\put(70,60){\circle*{3}}
\put(80,60){\circle*{3}}

\put(0,70){\textcolor{gray1}{\circle*{3}}}
\put(10,70){\textcolor{gray1}{\circle*{3}}}
\put(20,70){\textcolor{gray1}{\circle*{3}}}
\put(30,70){{\circle*{3}}}
\put(40,70){{\circle*{3}}}
\put(50,70){\circle*{3}}
\put(60,70){\circle*{3}} 
\put(70,70){\circle*{3}} 
\put(80,70){\circle*{3}}

\put(0,80){\textcolor{gray1}{\circle*{3}}}
\put(10,80){\textcolor{gray1}{\circle*{3}}}
\put(20,80){\textcolor{gray1}{\circle*{3}}}
\put(30,80){{\circle*{3}}}
\put(40,80){{\circle*{3}}}
\put(50,80){{\circle*{3}}}
\put(60,80){\circle*{3}} 
\put(70,80){\circle*{3}} 
\put(80,80){\circle*{3}}

\put(0,90){\textcolor{gray1}{\circle*{3}}}
\put(10,90){\textcolor{gray1}{\circle*{3}}}
\put(20,90){\textcolor{gray1}{\circle*{3}}}
\put(30,90){\textcolor{gray1}{\circle*{3}}}
\put(40,90){{\circle*{3}}}
\put(50,90){{\circle*{3}}}
\put(60,90){{\circle*{3}}}
\put(70,90){\circle*{3}} 
\put(80,90){\circle*{3}}

\end{picture}
\centering
\caption{The ideal $I$}\label{fig:2-ex3-}
\end{minipage}
\end{figure}
 
\end{example}

Inspired by the above example, we present the following theorem, which is the aim of this section. In this theorem, we show that non-surjectivity can occur for any saturated affine semigroup that is not isomorphic to the first octant $\ZZ^n_{\geq0}$. Recall that, unless stated otherwise, all our semigroups are affine, saturated, pointed and minimally embedded in $M$.

\begin{theorem} \label{th:surjective derivatio}
Let $S$ be a pointed saturated affine semigroup. The natural map 
$$\pi\colon \Der_I(\KK[S]) \to \Der(\KK[S]/I)$$ 
is surjective for every monomial ideal $I$ if and only if $S \cong \ZZ^n_{\geq 0}$.
\end{theorem}

If $\rank M = 1$, then the theorem holds, since every pointed saturated semigroup is isomorphic to $\ZZ_{\ge 0}$ in this case. Hence, in the sequel we assume $\rank M > 1$.

We divide the proof into several lemmas. Let $S$ be a saturated affine semigroup. We say that $m \in S \setminus \{0\}$ is irreducible if $m = m_1 + m_2$ implies $m_1 = 0$ or $m_2 = 0.$ The set $\mathcal{H}$ of all irreducible elements of $S$ is called the Hilbert basis of $S$. Since $S$ is pointed, the Hilbert basis $\mathcal{H}$ is the unique minimal generating set of $S$ as a semigroup. Let $\mathcal{H} = \{\mu_1, \dots, \mu_l\} \subset S$ be the Hilbert basis of $S$. In all cases where $S$ is different from the first octant $\ZZ_{\geq 0}^n$, we will prove that there exists a non-liftable derivation in $\KK[S]/I_{\mathcal{H}}$, where $I_{\mathcal{H}}$ is the monomial ideal 
$$I_{\mathcal{H}} = \bigoplus_{m \in S \setminus \mathcal{H}} \KK \cdot \XX^m\,.$$

Let $\rho \in S^\vee(1)$. Note that there exist $m_1, m_2 \in \mathcal{H}$ such that $\rho(m_1) = 0$ and $\rho(m_2) = 1$, since $\mathcal{H}$ generates $S$.

\begin{remark}
If the natural map $\pi\colon\operatorname{Der}_I(\mathbf{k}[S])\to \operatorname{Der}(\mathbf{k}[S]/I)$ is surjective, then every homogeneous derivation in $\operatorname{Der}(\mathbf{k}[S]/I)$
 has a homogeneous preimage with the same degree in $\operatorname{Der}_I(\mathbf{k}[S])$
\end{remark}

\begin{lemma}\label{lemma: nor sujective l-index}
Let $S$ be a semigroup. If there exists $\rho\in S^\vee(1)$ such that $\{0,1\}$ is a proper subset of $\rho(\mathcal{H})$, then the natural map $\pi\colon \Der_{I_\mathcal{H}}(\KK[S])\to\Der(\KK[S]/I_\mathcal{H})$ is not surjective.
\end{lemma}
\begin{proof}
 Let $\mathcal{H}=\{\mu_1,\dots,\mu_l\}$ be the Hilbert basis of $S$. By hypothesis, there exists $\rho\in S^{\vee}(1)$ such that the set $\{0,1\}$ is a proper subset of $\rho(\mathcal{H}).$ Without loss of generality, we assume that $\rho(\mu_1)=\ell>1$ and $\rho(\mu_2)=0$. We define the linear map
$$\overline{\partial}\colon \KK[S]/I_{\mathcal{H}}\to\KK[S]/I_{\mathcal{H}},\quad\mbox{given by}\quad \overline{\partial}(\overline{\XX}^{0})=0,\ \overline{\partial}(\overline{\XX}^{\mu_1})=\overline{\XX}^{\mu_2},\mbox{ and }\overline{\partial}(\overline{\XX}^{\mu_i})=0\mbox{ for all } i>1\,.$$
The map $\overline{\partial}$ is well defined since $\overline{\partial}(I_\mathcal{H})=0$. Moreover, 
it is a derivation. Indeed, the only non-trivial instance of the Leibniz rule is 
$$0=\overline{\partial}(\overline{\XX}^{\mu_1+\mu_i})=\overline{\partial}(\overline{\XX}^{\mu_1})\overline{\XX}^{\mu_i}+\overline{\XX}^{\mu_1}\overline{\partial}(\overline{\XX}^{\mu_i})=\overline{\XX}^{\mu_1}\overline{\partial}(\overline{\XX}^{\mu_i})\,.$$
Note that $\overline{\partial}(\overline{\XX}^{\mu_i})$ is equal to $0$ or $\overline{\XX}^{\mu_2}$. In both cases, the Leibniz rule is verified since $$\overline{\XX}^{\mu_1}\overline{\partial}(\overline{\XX}^{\mu_i})=0\,.$$ 

The derivation $\overline{\partial}$ is homogeneous and its degree is $\deg\overline{\partial}=\alpha=\mu_2-\mu_1$ and $\rho(\alpha)=-\ell<-1$. Assume that there exists $\partial\in \Der(\KK[S])$ homogeneous such that $\pi(\partial)=\overline{\partial}$. The degrees of $\partial$ and $\overline{\partial}$ are both $\alpha$. Since $\rho(\alpha)$ is negative, $\partial$ must be an outer derivation. Finally, by \eqref{eq:R} we have $\rho(\alpha)\geq -1$. This provides a contradiction. Hence, the lifting $\partial$ does not exist. 
\end{proof}

In the proof of \cref{lemma: proper subset general}, we need the following technical lemma.

\begin{lemma}\label{lemma: Det(A)}
 Let $M$ be a free abelian group of rank $n$ and let $N$ be the dual group $N=\homo(M,\ZZ)$. Let $\beta=\{\mu_1,\ldots,\mu_n\}\subset M$ and $\beta'=\{\rho_1,\ldots,\rho_n\}\subset N$ be linearly independent sets. We let $M_\beta$ be the subgroup of $M$ generated by $\beta$ and $N_{\beta'}$ be the subgroup of $N$ generated by $\beta'$. If $A=(a_{ij})$ is the $n\times n$ matrix with $a_{ij}=\rho_i(\mu_j)$, then 
 $$|M/M_\beta|\cdot |N/N_{\beta'}|=|\det(A)|\,.$$

\end{lemma}

\begin{proof}
 The lemma follows directly by taking dual bases of $M$ and $N$, respectively and applying \cite[Chapter~VII, Section~2, Theorem~2.5]{barvinok} to the subgroups $M_\beta\subset M$ and $N_{\beta'}\subset N$ on both sides of the duality.
\end{proof}

To prove \cref{th:surjective derivatio}, we introduce the following notation that will be convenient in the proof.

\begin{definition}
 Let $S$ be a semigroup minimally embedded in $M$ with $\rank M = n$. We define a complete flag $\mathcal{F}$ of $S$ as a chain of faces $F_i \subseteq S$ such that
 $$\{0\} = F_0 \subset F_1 \subset \ldots \subset F_n = S \quad \text{with} \quad \rank F_i = i\,.$$
 The complete flag $\mathcal{F}$ also defines a chain of reverse inclusions 
 $$S^\vee = F^*_0 \supset F^*_1 \supset \ldots \supset F^*_n = \{0\} \quad \text{with} \quad \rank F^*_i = n - i\,,$$
 where $F^*_i$ is the face of $S^\vee \subset N$ dual to $F_i$, defined as $F^*_i = S^\vee \cap F_i^\perp$.
 
 Given a complete flag, an upper triangular pair of collections of rays is a pair of sets $\beta = \{\mu_1, \ldots, \mu_n\} \subset S(1)$ and $\beta' = \{\rho_1, \ldots, \rho_n\} \subset S^\vee(1)$ such that 
 $$\mu_i \in F_i \setminus F_{i-1} \quad \text{and} \quad \rho_i \in F^*_{i-1} \setminus F^*_i\,,$$
 satisfying 
 \begin{align} \label{eq:matrix-A}
\rho_i(\mu_j) =
\begin{cases}
 0 & \text{if} \quad i > j \\
 * & \text{if} \quad i \leq j
\end{cases} \quad \text{and} \quad \rho_i(\mu_i) > 0, \text{ for all } i\,,
\end{align}
\end{definition}
where $*$ denotes any element in $\KK$. In the following lemma, we show that an upper triangular pair of collections of rays always exists.

\begin{lemma}\label{lemma: proper subset general}
Let $S$ be a semigroup and let $\mathcal{H}$ be its Hilbert basis. Let $\mathcal{F}$ be a complete flag of $S$. Then, there exists an upper triangular pair of collections of rays $(\beta, \beta')$. Moreover, if $\rho(S(1)) = \{0,1\}$ for all $\rho \in S^\vee(1)$, then $\beta$ and $\beta'$ are $\ZZ$-bases of $M$ and $N$, respectively.
\end{lemma}

\begin{proof}
Let $S$ be minimally embedded in $M$ with $\rank M = n$. The complete flag $\mathcal{F}$ provides us with the inclusions of faces
$$\{0\} = F_0 \subset F_1 \subset \ldots \subset F_n = S \quad \text{and} \quad S^\vee = F^*_0 \supset F^*_1 \supset \ldots \supset F^*_n = \{0\}\,.$$

We will appropriately choose the sets $\beta = \{\mu_1, \ldots, \mu_n\} \subset S(1)$ and $\beta' = \{\rho_1, \ldots, \rho_n\} \subset S^\vee(1)$ to obtain an upper triangular pair $(\beta, \beta')$. First, we pick $\mu_i$ as any ray in $F_i \setminus F_{i-1}$. Then, we choose $\rho_i$ to be any ray in $F^*_{i-1}$ with $\rho_i(\mu_i) > 0$. Such a ray must exist; otherwise, $\mu_i \in F_{i-1}$, which is a contradiction. This choice satisfies the condition in \eqref{eq:matrix-A}. Indeed, if $i > j$, then $\rho_i$ is a ray in $F^*_{i-1} \subseteq F^*_j$ and $\mu_j$ is a ray in $F_j$, so $\rho_i(\mu_j) = 0$. This proves the first statement of the lemma.

To prove the second statement, we assume, moreover, that $\rho(S(1)) = \{0,1\}$. Under these conditions, $\rho_i(\mu_i) > 0$ in \eqref{eq:matrix-A} implies $\rho_i(\mu_i) = 1$. Let $A = (a_{ij})$ be the $n \times n$ matrix with $a_{ij} = \rho_i(\mu_j)$. By \eqref{eq:matrix-A} and by the fact that $\rho_i(\mu_i) = 1$, we have that $\det(A) = 1$, so $\beta$ and $\beta'$ are linearly independent sets. Moreover, by \cref{lemma: Det(A)}, we have that $|M/M'| = |N/N'| = 1$, where $M'$ and $N'$ are the submodules spanned by $\beta$ and $\beta'$ in $M$ and $N$, respectively. We conclude that $\beta$ and $\beta'$ are $\ZZ$-bases of $M$ and $N$, respectively.
\end{proof}

\begin{corollary}\label{lemma proper subset}
Let $S$ be a simplicial affine semigroup. If $S$ is not isomorphic to $\ZZ_{\geq 0}^n$, then the natural map $\pi\colon \Der_{I_\mathcal{H}}(\KK[S]) \to \Der(\KK[S]/I_\mathcal{H})$ is not surjective.
\end{corollary}

\begin{proof}
 If $S$ is not isomorphic to $\ZZ_{\geq 0}^n$, then the set $S(1)$ is not a $\ZZ$-basis of $M$. Now, take any complete flag of $S$. According to \cref{lemma: proper subset general}, we can order the sets $S(1)$ and $S^\vee(1)$, both with $n$ elements, to form an upper triangular pair of collections of rays. Furthermore, by the second statement in \cref{lemma: proper subset general}, there exists $\rho \in S^\vee(1)$ such that $\{0,1\}$ is a proper subset of $\rho(S(1)) \subset \rho(\mathcal{H})$, since $S(1)$ is not a $\ZZ$-basis of $M$. Thus, the non-surjectivity of $\pi$ follows from \cref{lemma: nor sujective l-index}.
\end{proof}

In the following lemma, we handle the case where $S$ is not simplicial.

\begin{lemma}\label{lemma: nor sujective non-simplicial}
Let $S$ be a semigroup. If $S$ is not simplicial, then the natural map $\pi\colon \Der_{I_\mathcal{H}}(\KK[S]) \to \Der(\KK[S]/I_\mathcal{H})$ is not surjective.
\end{lemma}

\begin{proof} 
Let $F$ be a non-simplicial face of $S$ of minimal rank $r$. Note that $ 3 \leq r \leq n $.  We now choose faces $ F_1 $ and $ F_2 $ of $ F $, both of rank $ r-1 $, such that their intersection $ F_1 \cap F_2 $ has rank $r-2$, and such that $ F_1 \cup F_2$ does not contain all the rays in $S(1)$. We claim that such choice is possible. 

Indeed, let $G \subset F$ be a face of rank $r-2$, the dual face of $G$ in $F$ is $G^*\cap F^\bot$ and has rank 2. Hence, $G^*\cap F^\bot$ is generated by exactly two rays, say $ \rho_1 $ and $ \rho_2 $. The faces $ F_1 = \rho_1^\perp \cap F$ and $ F_2 = \rho_2^\perp \cap F $ are facets of $F$ such that $ G = F_1 \cap F_2 $. Since $ F_1 $ and $ F_2 $ have $ r-1 $ rays with $ r-2 $ rays in common, their union $ F_1 \cup F_2 $ has $ r $ rays. But since $ F $ is non-simplicial, there must exist a ray of $F$ that is not contained in $ F_1 \cup F_2 $.

Letting $P\subset S^{\vee}(1)$ be the set
\[
P = \{ \rho \in S^{\vee}(1) \mid \rho \notin F^* \}.
\]
We fix $ \rho_1, \rho_2 \in P $ such that $ \rho_1 \in F_1^* $, $ \rho_2 \in F_2^* $, and $ \rho_1 \neq \rho_2 $. This choice is possible since $ F^*$ is a face of each $ F_i^* $, and hence $ F_i^* \cap P \neq \emptyset $. Moreover, we have $ F_1^* \cap P \neq F_2^* \cap P $; otherwise, if $ F_1^* \cap P = F_2^* \cap P $, then for every $ \rho \in S^\vee(1) $, we would have $ \rho \in F_1^* $ if and only if $ \rho \in F_2^* $, implying $ F_1 = F_2 $, which contradicts the assumption that $ F_1 \neq F_2 $. Finally, observe that since $ \rho_i^* $ and $ F $ are both faces of $ S $, their intersection $ \rho_i^* \cap F $ is also a face of $ S $. Moreover, the rank of $ \rho_i^* \cap F $ is $r-1$. Hence, $\rho_i^* \cap F = F_i$.

For $ \mu_1 \in F \cap S$ with $ \mu_1 \notin F_1 \cup F_2 $, and $ \mu_2 \in F_1 \cap F_2 $, we define a derivation 
$$\overline{\partial}\colon \mathbf{k}[S]/I_\mathcal{H}\to \mathbf{k}[S]/I_\mathcal{H}\quad \mbox{given by}\quad 
\overline{\partial}(\overline{\mathbf{x}}^{\mu_1}) = \overline{\mathbf{x}}^{\mu_2}, \quad \text{and} \quad \overline{\partial}(\overline{\mathbf{x}}^{\mu}) = 0 \quad \text{for all } \mu \neq \mu_1.
$$
This defines a homogeneous linear map of degree $ \alpha = \mu_2 - \mu_1 $, which satisfies
\begin{align*} 
\rho_1(\alpha) = -\rho_1(\mu_1) < 0 \quad \text{and} \quad \rho_2(\alpha) = -\rho_2(\mu_1) < 0.
\end{align*}
Moreover, the linear map $\overline{\partial}$ is indeed a derivation where the only non-trivial instance of the Leibniz rule is 
$$0=\overline{\partial}(0)=\overline{\partial}(\overline{\mathbf{x}}^{\mu_i+\mu_j}) \quad \mbox{and}\quad\overline{\partial}(\overline{\mathbf{x}}^{\mu_i})\overline{\mathbf{x}}^{\mu_j}+\overline{\mathbf{x}}^{\mu_i}\overline{\partial}(\overline{\mathbf{x}}^{\mu_j})=0\,.$$

Indeed, $ \rho_i(\mu_1) > 0 $, since otherwise $ \rho_i(\mu_1) = 0 $ would imply that $\mu_1 \in \rho_i^* \cap F = F_i$, contradicting the assumption that $ \mu_1 \notin F_1 \cup F_2$.

Assume that there exists $\partial\in \Der(\KK[S])$ homogeneous such that $\pi(\partial)=\overline{\partial}$. Since $\rho_i(\alpha)$ is negative, $\partial$ must be an outer derivation. However, by \eqref{eq:R}, for every outer derivation, there exists a unique $ \rho \in S^\vee(1) $ such that $ \rho(\alpha) < 0 $, which contradicts the fact that both $ \rho_1(\alpha)$ and $ \rho_2(\alpha)$ are negative. This contradiction concludes the proof.
\end{proof}

\cref{lemma: nor sujective l-index} and \cref{lemma: nor sujective non-simplicial} handle all the cases other than the case where $\KK[S]$ is the polynomial ring. This case is contained in \cite[Lemma 2.1.2]{B95}. For the convenience of the reader, we provide a short argument. 

\begin{lemma}\label{lemma:affine-space}
Let $S = \ZZ_{\geq 0}^n$ and let $I$ be a monomial ideal of $\KK[S]$. Then the natural map $\pi\colon \Der_I(\KK[S]) \to \Der(\KK[S]/I)$ is surjective.
\end{lemma}

\begin{proof}
Recall that $\KK[S] = \KK[x_1, \dots, x_n]$ as in \cref{ex:polynomial-ring}. Let $\varphi\colon \KK[S] \to \KK[S]/I$ be the quotient map. Let $\overline\partial \in \Der(\KK[S]/I)$. We choose $f_i \in \varphi^{-1}(\overline{\partial}(\overline{x}_i))$ for all $1\leq i\leq n$. Since $\Der(\KK[S])$ is freely generated by the partial derivatives $\left\{\dfrac{d}{dx_1}, \dots, \dfrac{d}{dx_n}\right\}$ as a $\KK[S]$-module, see \cite[Section~3.2.2]{F17}, the map $\partial\colon \KK[S] \to \KK[S],$ defined by 
$$\partial = \sum_{i=1}^n f_i \dfrac{d}{dx_i} \quad \text{is a derivation on } \KK[S] \text{ such that } \partial(x_i) = f_i\,.$$
 Since $\overline{\partial} \circ \varphi = \varphi \circ \partial$, we have that $\partial(I) \subset I$. Hence $\pi(\partial) = \overline{\partial},$ which implies that $\pi$ is surjective.
\end{proof}

We can finally proceed with the proof of \cref{th:surjective derivatio}.

\begin{proof}[Proof of \cref{th:surjective derivatio}]
 The direct implication follows from \cref{lemma proper subset} if $S$ is simplicial and from \cref{lemma: nor sujective non-simplicial} if $S$ is not simplicial. On the other hand, the converse implication follows from \cref{lemma:affine-space}.
\end{proof}

\section{Liftable locally nilpotent derivation on monomial algebras}\label{section 3} 

Let $S$ be a semigroup and let $I \subset \KK[S]$ be a monomial ideal. In this section, we compute the homogeneous locally nilpotent derivations on the algebra $\KK[S]/I$, which are obtained as the images of derivations on $\KK[S]$ via $\pi$ in \eqref{projection}. We formalize this concept in the following definition.

\begin{definition}
 Let $S$ be a semigroup and let $I$ be a monomial ideal of $\KK[S]$. We say that a derivation $\overline{\partial}$ on $\KK[S]/I$ is \textit{liftable} if there exists a derivation $\partial \in \Der_I(\KK[S])$ such that $\overline{\partial} = \pi(\partial)$.
\end{definition}

\begin{remark}
  In view of \cref{th:surjective derivatio}, if $S$ is isomorphic to $\ZZ_{\geq 0}^n$, then all derivations on $\KK[S]/I$ are liftable.
\end{remark}

If $\overline{\partial}\colon\KK[S]/I\to \KK[S]/I$ is a liftable homogeneous derivation, then there exists a homogeneous derivation $\partial$ of $\KK[S]$ such that $\overline{\partial}=\pi(\partial)$. According to \cref{decomposition-derivations}, we have $\partial=\partial_{\alpha,p}.$ Therefore, we conclude that $\overline{\partial}=\pi(\partial_{\alpha,p})$, and we denote this derivation by $\overline{\partial}_{\alpha,p}$. Homogeneous derivations in $\KK[S]$ come in two types: inner and outer. The case of outer derivations on $\KK[S]$ is straightforward. Indeed, outer derivations are all locally nilpotent and so any such derivation that descends to the quotient $\KK[S]/I$ remains locally nilpotent. We formalize this in the following lemma for future reference.

\begin{lemma}\label{lemma: outlnd}
  Let $S$ be a saturated affine semigroup and let $I$ be a monomial ideal of $\KK[S].$ Then, any homogeneous derivation $\partial \in \mathfrak{s}(I)$ induces a locally nilpotent derivation on $\KK[S]/I$.
\end{lemma}

We now turn our attention to inner derivations. Inner derivations are never locally nilpotent on $\KK[S]$, but they may become so in the quotient $\KK[S]/I$. Recall that all inner derivations on $\KK[S]$ leave $I$ invariant and, therefore, pass to the quotient $\KK[S]/I$.

\begin{lemma}\label{lemma hlnd no empty}
  Let $I$ be a monomial ideal of $\KK[S]$ and let $\alpha \in S$. If $(\ZZ_{\geq 0} \cdot \alpha) \cap \supp(I) \neq \emptyset$, then the derivation $\quot{\partial}_{\alpha,p} \colon \KK[S]/I \to \KK[S]/I$ is locally nilpotent for all $p \in N_\KK$.
\end{lemma}

\begin{proof}
  By hypothesis, there exists $\ell \in \ZZ_{\geq 0}$ such that $\ell \cdot \alpha \in \supp(I)$. Let $m \in S$. A straightforward computation shows that
  \begin{align} \label{iteration}
    \partial_{\alpha,p}^\ell(\XX^m) = r \cdot \XX^{m+\ell \cdot \alpha} \quad \mbox{where} \quad r\in \KK\,.
  \end{align}


  Now, $\partial_{\alpha,p}^\ell(\XX^m) = r \cdot \XX^{m+\ell \cdot \alpha} = r \cdot \XX^m \cdot \XX^{\ell \cdot \alpha}$ and $\XX^{\ell \cdot \alpha} \in I$ since $\ell \cdot \alpha \in \supp(I)$. This yields $\partial_{\alpha,p}^\ell(\XX^m) \in I$ and so $\quot{\partial}_{\alpha,p}^\ell(\XX^m) = 0$, proving that $\quot{\partial}_{\alpha,p}$ is locally nilpotent.
\end{proof}

\begin{lemma}\label{proposition hlnd0}
Let $I$ be a monomial ideal of $\KK[S]$ and let $\alpha \in S$. Assume that $(\ZZ_{\geq 0} \cdot \alpha) \cap \supp(I) = \emptyset$. Then the derivation $\quot{\partial}_{\alpha,p} \colon \KK[S]/I \to \KK[S]/I$ is locally nilpotent if and only if
  \begin{align}\label{eq:cond-rara}
    (m + \ZZ_{\geq 0} \cdot \alpha) \cap \supp(I) \neq \emptyset \quad \mbox{for all} \quad m \in S \setminus p^\bot\,.
  \end{align}
\end{lemma}

\begin{proof}
  We claim that both propositions of the ``if and only if'' imply that $p(\alpha) = 0$. Indeed, condition \eqref{eq:cond-rara} applied to $\alpha$ implies $\alpha \in p^\bot$. On the other hand, if $\quot{\partial}_{\alpha,p}$ is locally nilpotent, then by \eqref{iteration} we have $\partial_{\alpha,p}^\ell(\XX^\alpha) = r \cdot \XX^{(\ell+1) \cdot \alpha}$ with $r \neq 0$ whenever $p(\alpha) \neq 0$, and since $(\ell+1)\cdot\alpha \notin \supp(I)$, this expression belongs to the ideal only if $p(\alpha) = 0$.

  Now, since $p(\alpha) = 0$, we have that 
  \begin{align}\label{iteration2} 
    \partial_{\alpha,p}^\ell(\XX^m) = p(m)^\ell \cdot \XX^{m + \ell \cdot \alpha}\quad \mbox{for all} \quad m\in S\,.
  \end{align}
  The derivation $\quot{\partial}_{\alpha,p}$ is locally nilpotent if and only if for every $m \in S$, there exists $\ell\in \ZZ_{\geq 0}$ such that $\partial_{\alpha,p}^\ell(\XX^m) \in I$. By \eqref{iteration2} and the same argument in \cref{lemma hlnd no empty}, this is the case if and only if
  \begin{align*}
    m + \ell \cdot \alpha \in \supp(I) \mbox{ for some } \ell \in \ZZ_{\geq 0} \quad \mbox{or} \quad p(m) = 0\,,
  \end{align*}
  and this last condition is equivalent to \eqref{eq:cond-rara}.
\end{proof}

\begin{remark}
  It is sufficient to verify the condition in \cref{proposition hlnd0} on a set of generators of $S$.
\end{remark}

We now state our main classification result for liftable homogeneous locally nilpotent derivations on $\KK[S]/I$. It amounts to the recollection of the three previous lemmas.

\begin{theorem} \label{main-classification}
  Let $S$ be a semigroup, and let $I$ be a monomial ideal on $\KK[S]$. A liftable homogeneous derivation $\overline{\partial}_{\alpha,p} \colon \KK[S]/I \to \KK[S]/I$ with $p \in N_{\KK}$ and $\alpha \in M$ is locally nilpotent if and only if
  \begin{enumerate}
    \item [$(i)$] $\alpha \in \mathcal{R}_\rho(I)$ and $p = \lambda \cdot \rho$ with $\rho\in S^\vee(1)$ and $\lambda \in \KK$; or
    \item [$(ii)$] $\alpha \in S$ and $(\ZZ_{\geq 0} \cdot \alpha) \cap \supp(I) \neq \emptyset$; or
    \item [$(iii)$] $\alpha \in S$, $(\ZZ_{\geq 0} \cdot \alpha) \cap \supp(I) = \emptyset$, and $(m + \ZZ_{\geq 0} \cdot \alpha) \cap \supp(I) \neq \emptyset$ for all $m \in S \setminus p^\bot$.
  \end{enumerate}
\end{theorem}

\begin{proof}
  The result follows directly from Lemmas \cref{lemma: outlnd}, \cref{lemma hlnd no empty}, and \cref{proposition hlnd0}.

\end{proof}

\begin{example}\label{Example:degree} 
With the notation in \cref{ex:polynomial-ring} for $n=2$, we present here examples of all the degrees of homogeneous derivations whose images in the quotient are locally nilpotent derivation for different monomial ideals $I$. Red points indicate the degrees described in \cref{main-classification}~$(i)$. Green points indicate the degrees described in \cref{main-classification}~$(ii)$. Finally, blue points indicate the degrees described in \cref{main-classification}~$(iii)$.
\begin{figure}[H]
\begin{minipage}[c]{0.3\textwidth}
\centering
\begin{picture}(100,95)
\definecolor{gray1}{gray}{0.7}
\definecolor{gray2}{gray}{0.85}
\definecolor{green}{RGB}{0,124,0}

\textcolor{gray2}{\put(0,20){\vector(1,0){90}}}
\textcolor{gray2}{\put(10,10){\vector(0,1){88}}}

\put(0,10){\textcolor{gray1}{\circle*{3}}}
\put(10,10){\textcolor{gray1}{\circle*{3}}}
\put(20,10){\textcolor{gray1}{\circle*{3}}} 
\put(30,10){\textcolor{gray1}{\circle*{3}}} 
\put(40,10){\textcolor{gray1}{\circle*{3}}} 
\put(50,10){\textcolor{gray1}{\circle*{3}}} 
\put(60,10){\textcolor{gray1}{\circle*{3}}} 
\put(70,10){\textcolor{gray1}{\circle*{3}}} 
\put(80,10){\textcolor{gray1}{\circle*{3}}}

\put(0,20){\textcolor{gray1}{\circle*{3}}}
\put(10,20){\circle{3}}
\put(20,20){\textcolor{blue}{\circle*{3}}}
\put(30,20){\textcolor{blue}{\circle*{3}}}
\put(40,20){\textcolor{blue}{\circle*{3}}}
\put(50,20){\textcolor{blue}{\circle*{3}}}
\put(60,20){\textcolor{blue}{\circle*{3}}}
\put(70,20){\textcolor{blue}{\circle*{3}}}
\put(80,20){\textcolor{blue}{\circle*{3}}}

\put(0,30){\textcolor{gray1}{\circle*{3}}}
\put(10,30){{\circle{3}}}
\put(20,30){\textcolor{green}{\circle*{3}}} 
\put(30,30){\textcolor{green}{\circle*{3}}} 
\put(40,30){\textcolor{green}{\circle*{3}}} 
\put(50,30){\textcolor{green}{\circle*{3}}} 
\put(60,30){{\circle*{3}}} 
\put(70,30){\circle*{3}}
\put(80,30){\circle*{3}}

\put(0,40){\textcolor{gray1}{\circle*{3}}}
\put(10,40){{\circle{3}}}
\put(20,40){\textcolor{green}{\circle*{3}}} 
\put(30,40){\textcolor{green}{\circle*{3}}} 
\put(40,40){{\circle*{3}}}
\put(50,40){{\circle*{3}}}
\put(60,40){\circle*{3}} 
\put(70,40){\circle*{3}}
\put(80,40){\circle*{3}}

\put(0,50){\textcolor{gray1}{\circle*{3}}}
\put(10,50){{\circle{3}}}
\put(20,50){\textcolor{green}{\circle*{3}}} 
\put(30,50){\textcolor{green}{\circle*{3}}} 
\put(40,50){\circle*{3}}
\put(50,50){\circle*{3}}
\put(60,50){\circle*{3}} 
\put(70,50){\circle*{3}} 
\put(80,50){\circle*{3}}

\put(0,60){\textcolor{gray1}{\circle*{3}}}
\put(10,60){{\circle{3}}}
\put(20,60){\textcolor{green}{\circle*{3}}} 
\put(30,60){\textcolor{green}{\circle*{3}}} 
\put(40,60){\circle*{3}}
\put(50,60){\circle*{3}}
\put(60,60){\circle*{3}} 
\put(70,60){\circle*{3}}
\put(80,60){\circle*{3}}

\put(0,70){\textcolor{gray1}{\circle*{3}}}
\put(10,70){{\circle{3}}}
\put(20,70){\textcolor{green}{\circle*{3}}} 
\put(30,70){{\circle*{3}}}
\put(40,70){{\circle*{3}}}
\put(50,70){\circle*{3}}
\put(60,70){\circle*{3}} 
\put(70,70){\circle*{3}} 
\put(80,70){\circle*{3}}

\put(0,80){\textcolor{gray1}{\circle*{3}}}
\put(10,80){{\circle{3}}}
\put(20,80){\textcolor{green}{\circle*{3}}} 
\put(30,80){{\circle*{3}}}
\put(40,80){{\circle*{3}}}
\put(50,80){{\circle*{3}}}
\put(60,80){\circle*{3}} 
\put(70,80){\circle*{3}} 
\put(80,80){\circle*{3}}

\put(0,90){\textcolor{gray1}{\circle*{3}}}
\put(10,90){{\circle{3}}}
\put(20,90){\textcolor{green}{\circle*{3}}} 
\put(30,90){{\circle*{3}}}
\put(40,90){{\circle*{3}}}
\put(50,90){{\circle*{3}}}
\put(60,90){{\circle*{3}}}
\put(70,90){\circle*{3}} 
\put(80,90){\circle*{3}}

\end{picture}
\centering
\caption{$ $}
{$I_1=(x^2y^5,x^3y^2,x^5y)$}\label{fig:1-ex4}
\end{minipage}
 \hfill
\begin{minipage}[c]{0.3\textwidth}
\centering
\begin{picture}(100,95)
\definecolor{gray1}{gray}{0.7}
\definecolor{gray2}{gray}{0.85}
\definecolor{green}{RGB}{0,124,0}

\textcolor{gray2}{\put(0,20){\vector(1,0){90}}}
\textcolor{gray2}{\put(10,10){\vector(0,1){88}}}

\put(0,10){\textcolor{gray1}{\circle*{3}}}
\put(10,10){\textcolor{gray1}{\circle*{3}}}
\put(20,10){\textcolor{gray1}{\circle*{3}}} 
\put(30,10){\textcolor{red}{\circle*{3}}} 
\put(40,10){\textcolor{red}{\circle*{3}}} 
\put(50,10){\textcolor{red}{\circle*{3}}} 
\put(60,10){\textcolor{red}{\circle*{3}}} 
\put(70,10){\textcolor{red}{\circle*{3}}} 
\put(80,10){\textcolor{red}{\circle*{3}}}

\put(0,20){\textcolor{gray1}{\circle*{3}}}
\put(10,20){\circle{3}} 
\put(20,20){\textcolor{green}{\circle*{3}}} 
\put(30,20){\textcolor{green}{\circle*{3}}} 
\put(40,20){\textcolor{green}{\circle*{3}}} 
\put(50,20){\textcolor{green}{\circle*{3}}} 
\put(60,20){\circle*{3}}
\put(70,20){\circle*{3}}
\put(80,20){\circle*{3}}

\put(0,30){\textcolor{gray1}{\circle*{3}}}
\put(10,30){\textcolor{blue}{\circle*{3}}} 
\put(20,30){\textcolor{green}{\circle*{3}}} 
\put(30,30){\textcolor{green}{\circle*{3}}} 
\put(40,30){\textcolor{green}{\circle*{3}}} 
\put(50,30){\textcolor{green}{\circle*{3}}} 
\put(60,30){{\circle*{3}}} 
\put(70,30){\circle*{3}}
\put(80,30){\circle*{3}}

\put(0,40){\textcolor{gray1}{\circle*{3}}}
\put(10,40){\textcolor{blue}{\circle*{3}}} 
\put(20,40){\textcolor{green}{\circle*{3}}} 
\put(30,40){\textcolor{green}{\circle*{3}}} 
\put(40,40){{\circle*{3}}}
\put(50,40){{\circle*{3}}}
\put(60,40){\circle*{3}} 
\put(70,40){\circle*{3}}
\put(80,40){\circle*{3}}

\put(0,50){\textcolor{gray1}{\circle*{3}}}
\put(10,50){\textcolor{blue}{\circle*{3}}} 
\put(20,50){\textcolor{green}{\circle*{3}}} 
\put(30,50){\textcolor{green}{\circle*{3}}} 
\put(40,50){\circle*{3}}
\put(50,50){\circle*{3}}
\put(60,50){\circle*{3}} 
\put(70,50){\circle*{3}} 
\put(80,50){\circle*{3}}

\put(0,60){\textcolor{gray1}{\circle*{3}}}
\put(10,60){\textcolor{blue}{\circle*{3}}} 
\put(20,60){\textcolor{green}{\circle*{3}}} 
\put(30,60){\textcolor{green}{\circle*{3}}} 
\put(40,60){\circle*{3}}
\put(50,60){\circle*{3}}
\put(60,60){\circle*{3}} 
\put(70,60){\circle*{3}}
\put(80,60){\circle*{3}}

\put(0,70){\textcolor{gray1}{\circle*{3}}}
\put(10,70){\textcolor{blue}{\circle*{3}}} 
\put(20,70){{\circle*{3}}}
\put(30,70){{\circle*{3}}}
\put(40,70){{\circle*{3}}}
\put(50,70){\circle*{3}}
\put(60,70){\circle*{3}} 
\put(70,70){\circle*{3}} 
\put(80,70){\circle*{3}}

\put(0,80){\textcolor{gray1}{\circle*{3}}}
\put(10,80){\textcolor{blue}{\circle*{3}}} 
\put(20,80){{\circle*{3}}}
\put(30,80){{\circle*{3}}}
\put(40,80){{\circle*{3}}}
\put(50,80){{\circle*{3}}}
\put(60,80){\circle*{3}} 
\put(70,80){\circle*{3}} 
\put(80,80){\circle*{3}}

\put(0,90){\textcolor{gray1}{\circle*{3}}}
\put(10,90){\textcolor{blue}{\circle*{3}}} 
\put(20,90){{\circle*{3}}}
\put(30,90){{\circle*{3}}}
\put(40,90){{\circle*{3}}}
\put(50,90){{\circle*{3}}}
\put(60,90){{\circle*{3}}}
\put(70,90){\circle*{3}} 
\put(80,90){\circle*{3}}

\end{picture}
\centering
\caption{$ $}
{$I_2=(xy^5,x^3y^2,x^5)$}\label{fig:2-ex4}
\end{minipage}
 \hfill
\begin{minipage}[c]{0.3\textwidth}
\centering
\begin{picture}(100,95)
\definecolor{gray1}{gray}{0.7}
\definecolor{gray2}{gray}{0.85}
\definecolor{green}{RGB}{0,124,0}

\textcolor{gray2}{\put(0,20){\vector(1,0){90}}}

\textcolor{gray2}{\put(10,10){\vector(0,1){88}}}

\put(0,10){\textcolor{gray1}{\circle*{3}}}
\put(10,10){\textcolor{gray1}{\circle*{3}}}
\put(20,10){\textcolor{gray1}{\circle*{3}}} 
\put(30,10){\textcolor{gray1}{\circle*{3}}} 
\put(40,10){\textcolor{red}{\circle*{3}}} 
\put(50,10){\textcolor{red}{\circle*{3}}} 
\put(60,10){\textcolor{red}{\circle*{3}}} 
\put(70,10){\textcolor{red}{\circle*{3}}} 
\put(80,10){\textcolor{red}{\circle*{3}}}

\put(0,20){\textcolor{gray1}{\circle*{3}}}
\put(10,20){\circle{3}} 
\put(20,20){\textcolor{green}{\circle*{3}}} 
\put(30,20){\textcolor{green}{\circle*{3}}} 
\put(40,20){\textcolor{green}{\circle*{3}}} 
\put(50,20){\textcolor{green}{\circle*{3}}} 
\put(60,20){\circle*{3}}
\put(70,20){\circle*{3}}
\put(80,20){\circle*{3}}

\put(0,30){\textcolor{gray1}{\circle*{3}}}
\put(10,30){\textcolor{green}{\circle*{3}}} 
\put(20,30){\textcolor{green}{\circle*{3}}} 
\put(30,30){\textcolor{green}{\circle*{3}}} 
\put(40,30){\textcolor{green}{\circle*{3}}} 
\put(50,30){\textcolor{green}{\circle*{3}}} 
\put(60,30){{\circle*{3}}} 
\put(70,30){\circle*{3}}
\put(80,30){\circle*{3}}

\put(0,40){\textcolor{gray1}{\circle*{3}}}
\put(10,40){\textcolor{green}{\circle*{3}}} 
\put(20,40){\textcolor{green}{\circle*{3}}} 
\put(30,40){\textcolor{green}{\circle*{3}}} 
\put(40,40){{\circle*{3}}}
\put(50,40){{\circle*{3}}}
\put(60,40){\circle*{3}} 
\put(70,40){\circle*{3}}
\put(80,40){\circle*{3}}

\put(0,50){\textcolor{red}{\circle*{3}}}
\put(10,50){\textcolor{green}{\circle*{3}}} 
\put(20,50){\textcolor{green}{\circle*{3}}} 
\put(30,50){\textcolor{green}{\circle*{3}}} 
\put(40,50){\circle*{3}}
\put(50,50){\circle*{3}}
\put(60,50){\circle*{3}} 
\put(70,50){\circle*{3}} 
\put(80,50){\circle*{3}}

\put(0,60){\textcolor{red}{\circle*{3}}}
\put(10,60){\textcolor{green}{\circle*{3}}} 
\put(20,60){\textcolor{green}{\circle*{3}}} 
\put(30,60){\textcolor{green}{\circle*{3}}} 
\put(40,60){\circle*{3}}
\put(50,60){\circle*{3}}
\put(60,60){\circle*{3}} 
\put(70,60){\circle*{3}}
\put(80,60){\circle*{3}}

\put(0,70){\textcolor{red}{\circle*{3}}}
\put(10,70){{\circle*{3}}}
\put(20,70){{\circle*{3}}}
\put(30,70){{\circle*{3}}}
\put(40,70){{\circle*{3}}}
\put(50,70){\circle*{3}}
\put(60,70){\circle*{3}} 
\put(70,70){\circle*{3}} 
\put(80,70){\circle*{3}}

\put(0,80){\textcolor{red}{\circle*{3}}}
\put(10,80){{\circle*{3}}}
\put(20,80){{\circle*{3}}}
\put(30,80){{\circle*{3}}}
\put(40,80){{\circle*{3}}}
\put(50,80){{\circle*{3}}}
\put(60,80){\circle*{3}} 
\put(70,80){\circle*{3}} 
\put(80,80){\circle*{3}}

\put(0,90){\textcolor{red}{\circle*{3}}}
\put(10,90){{\circle*{3}}}
\put(20,90){{\circle*{3}}}
\put(30,90){{\circle*{3}}}
\put(40,90){{\circle*{3}}}
\put(50,90){{\circle*{3}}}
\put(60,90){{\circle*{3}}}
\put(70,90){\circle*{3}} 
\put(80,90){\circle*{3}}

\end{picture} 
\centering
\caption{$ $}
{$I_3=(y^5,x^3y^2,x^5)$}\label{fig:3-ex2}
\end{minipage}

\end{figure}
\end{example}

Some derivations appearing in \cref{main-classification} are zero derivations. In the following lemma, we provide a criterion to determine when a homogeneous derivation on $\KK[S]$ is the zero derivation on $\KK[S]/I$.

\begin{lemma} \label{lemma: trivial-derivation}
 A homogeneous derivation $\overline{\partial}_{\alpha,p}\colon \KK[S]/I\to \KK[S]/I$ with $p\in N_{\KK}$ and $\alpha\in M$ is zero if and only if 
 $$\big((S\setminus p^\bot)+\alpha\big)\subset \supp(I)\,.$$
 \end{lemma}

\begin{proof}
 The derivation $\partial=\partial_{\alpha,p}$ maps to the zero derivation via $\pi$ if and only if for every $m\in S$ we have $\partial(\XX^m)\in I$. This holds if and only if for every $m\in S$ with $p(m)\neq 0$, we have $m+\alpha$ contained in $\supp(I)$, so $\XX^{m+\alpha}\in I$. This last statement is equivalent to the one given in the lemma. 
\end{proof}

\begin{remark}
 It is sufficient to verify the condition in \cref{lemma: trivial-derivation} on a set of generators of $S$.
\end{remark}

\section{Automorphism group of monomial algebras}\label{section 4}

In this section, we apply our classification given in \cref{main-classification} to describe $\operatorname{Aut}_\KK(\KK[S]/I)$ in the case where $I$ is a monomial ideal with cofinite support and $S$ is the first octant. Our approach is inspired by Demazure's description of $\Aut(X)$ for a complete smooth toric variety $X,$ see \cite{demazure1970sous}, which was revisited in modern terms by several of the authors \cite{liendo2010affine,LL21, DL24}, see also \cite{C95}. In the following proposition, we prove that the automorphism group $\Aut_\KK(\KK[S]/I)$ is linear algebraic. The proposition is well known \cite[Section~7.6, Exercise~3]{H75}, for the convenience of the reader, we provide a short proof.

\begin{proposition}\label{propositio: generating}
Let $A$ be a $\KK$-algebra that is finite-dimensional as $\KK$-vector space. Then the automorphism group of $A$ is linear algebraic.
\end{proposition}

\begin{proof}
Since $A$ is finite-dimensional as vector space over $\KK$ and every automorphism $\varphi\colon A \to A$ is also an invertible linear map, we have that $\Aut_\KK(A)$ embeds in $\operatorname{GL}(A)$. Moreover, $\varphi \in \operatorname{GL}(A)$ belongs to $\Aut_\KK(A)$ if and only if $\varphi(fg)=\varphi(f) \cdot \varphi(g)$ for all $f,g$ in a $\KK$-basis of $A$, which is a closed condition. We conclude that $\Aut_\KK(A)$ is a closed subgroup of $\operatorname{GL}(A)$ and thus is linear algebraic.
\end{proof}

Let now $T=\spec \KK[M]$. Then, $T$ acts faithfully on $\KK[S]$ via the usual toric action given by the comorphism
$$\KK[S] \to \KK[M] \otimes \KK[S] \quad \text{defined by} \quad \XX^m \mapsto \XX^m \otimes \XX^m.$$
A closed point $t \in T$ corresponds to a maximal ideal of $\KK[M]$ and, in turn, corresponds to a homomorphism $t\colon M \rightarrow \KK^*$. Indeed, $t$ induces a homomorphism $\KK[M] \rightarrow \KK$ via $\XX^m \mapsto t(m)$, and the maximal ideal defining the point $t$ is the kernel of this homomorphism. With this notation, the action of a closed point $t \in T$ on $\KK[S]$ is given by 
\begin{align} \label{eq:torus-action}
\KK[S] \to \KK[S] \quad \text{defined by} \quad \XX^m \mapsto t \cdot \XX^m = t(m)\XX^m.
\end{align} 
In particular, the $T$-action preserves every monomial ideal $I$ and thus defines a $T$-action on $\KK[S]/I$.

\begin{definition} \label{def:full}
Let $S$ be a semigroup and let $I$ be a monomial ideal of $\KK[S]$. We say that $I$ is \emph{full} if no ray generator of $S$ is contained in $\supp(I)$. 
\end{definition}

In the case that a monomial ideal $I$ in $\KK[S]$ is not full, let $S'$ be the smallest subsemigroup of $S$ containing $S \setminus \supp(I)$, and let $I' = I \cap \KK[S']$. Then, we have that $\KK[S]/I = \KK[S']/I'$ and $I'$ is full.

\begin{example}
If $S$ is the first octant as in \cref{ex:polynomial-ring}, then the action of $T$ on $\KK[S]/I$ is given by component-wise multiplication. Furthermore, a monomial ideal $I \subset \KK[S]$ is full if and only if no variable $x_i$ belongs to $I$.
\end{example}

\begin{lemma} \label{lem:torus-faithful}
Let $S$ be a semigroup and let $I$ be a monomial ideal of $\KK[S].$ 
If $S \setminus \supp(I)$ generates $M$ as a group, 
then $T$ acts faithfully on $\KK[S]/I$.
\end{lemma}

\begin{proof}
Let $t \in T$ act as the identity on $\KK[S]/I$. By \eqref{eq:torus-action}, 
we have $t\cdot\XX^m = t(m)\XX^m = \XX^m$ for all $m \in S \setminus \supp(I)$. 
Thus, $t(m) = 1$ for all $m \in S \setminus \supp(I)$.
Since $S \setminus \supp(I)$ generates $M$ as a group,
$t(m) = 1$ for all $m \in M$, and hence $t$ is the identity in $T$.
\end{proof}

When $S$ is the first octant and $I$ is full, the standard basis vectors $e_1,\dots,e_n$ lie in $S\setminus\supp(I)$ and generate $M$, so \cref{lem:torus-faithful} applies and $T$ acts faithfully on $\KK[S]/I$. We denote the image of $T$ inside $\Aut_\KK(\KK[S]/I)$ also by $T$. In \cref{lemma: torus}, we show that $T$ is moreover a maximal torus.

\begin{lemma}\label{lemma: torus}
Let $S$ be the first octant and let $I$ be a full monomial ideal of $\KK[S]$ with cofinite support. We also let $T=\spec \KK[M]\subset \Aut_\KK(\KK[S]/I)$. Then, the centralizer of $T$ equals $T$. In particular, $T$ is a maximal torus of $\Aut_\KK(\KK[S]/I)$.
\end{lemma}

\begin{proof}
Let $g\in \Aut_\KK(\KK[S]/I)$ be defined by 
$$g(\overline{\XX}^{m'})=\sum_{m\in S\setminus \supp(I)} a_{m',m}\cdot\overline{\XX}^m\quad\mbox{for all}\quad m'\in S\setminus \supp(I)\,.$$ 
Assume that $g$ belongs to the centralizer of $T$, that is, $t\circ g(\overline{\XX}^{m'})=g\circ t(\overline{\XX}^{m'})$, for all $t\in T$. This implies 
$$\sum_{m\in S\setminus \supp(I)} t(m)\cdot a_{m',m}\cdot\overline{\XX}^m= \sum_{m\in S\setminus \supp(I)} t(m')\cdot a_{m',m}\cdot\overline{\XX}^m, \quad\mbox{for all }t\in T\,.$$
This equality can only hold if $a_{m',m}=0$ for all $m\neq m'$. Letting $a_{m,m}=a_{m}$, we conclude 
\begin{align} \label{eq:no-bigger}
g(\overline{\XX}^{m})=a_{m}\overline{\XX}^{m}\quad\mbox{for all}\quad m\in S\setminus \supp(I)\,.
\end{align} 
Finally, since the basis vector $e_i$ belongs to $S\setminus \supp(I)$ as $S$ is the first octant and $I$ is full, for every $m=(m_1,\ldots,m_n)\in M$, we have that $a_m=\prod a_{e_i}^{m_i}$ and so $g\in T$ by \eqref{eq:torus-action}. This proves, in particular, that the centralizer of $T$ is $T$ itself, thus $T$ is a maximal torus of $\Aut_\KK(\KK[S]/I)$.
\end{proof}

The hypothesis that $S$ is the first octant in the above lemma is essential, as \cref{example: non maximal torus} below shows. Due to this, in what follows we will assume that $S$ is the first octant, unless otherwise specified.

\begin{example}\label{example: non maximal torus}
Let $S$ be the semigroup $\mathbb{Z}_{\geq 0}\{(1,0),(1,1),(1,2)\}$. We consider the monomial ideal $I \subset \KK[S]$ given by $I=(\XX^{(2,0)},\XX^{(2,1)},\XX^{(2,2)},\XX^{(2,3)},\XX^{(2,4)})$. Thus, $\KK[S]/I \cong \KK \cdot \overline{\XX}^{(0,0)} \oplus \KK \cdot \overline{\XX}^{(1,0)} \oplus \KK \cdot \overline{\XX}^{(1,1)} \oplus \KK \cdot \overline{\XX}^{(1,2)}$. If $g$ is an element of the centralizer of $(\KK^*)^2$, then, using the same argument as in the proof of \cref{lemma: torus} up to \eqref{eq:no-bigger}, we conclude that
$$ g(\overline{\XX}^{(1,0)}) = a_{(1,0)}\overline{\XX}^{(1,0)}, \quad g(\overline{\XX}^{(1,1)}) = a_{(1,1)}\overline{\XX}^{(1,1)}, \quad \text{and} \quad g(\overline{\XX}^{(1,2)}) = a_{(1,2)}\overline{\XX}^{(1,2)}. $$
However, these three coefficients $a_{(1,0)}$, $a_{(1,1)}$, and $a_{(1,2)}$ are algebraically independent since the relation $ a_{(1,0)} \cdot a_{(1,2)} = a^2_{(1,1)} $ does not hold in the quotient. This implies that there is a 3-dimensional torus acting faithfully on $\KK[S]/I$ that is an extension of $T$. In particular, the centralizer of $T$ is larger than $T$.

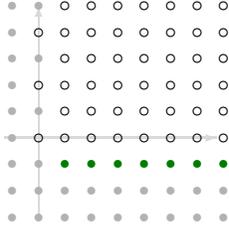
Furthermore, a straightforward verification shows that $\KK[S]/I$ is isomorphic to $\KK[S']/I'$ with $S'$ being the first octant in $M' = \mathbb{Z}^3$ and $I' = (\XX^{(2,0,0)}, \XX^{(0,2,0)}, \XX^{(0,0,2)}, \XX^{(1,1,0)}, \XX^{(1,0,1)}, \XX^{(0,1,1)})$. Now, \cref{lemma: torus} shows that a maximal torus of $\Aut_\KK(\KK[S']/I') \cong \Aut_\KK(\KK[S]/I)$ is indeed 3-dimensional.
\end{example}

\begin{definition}\label{def:root-subgroup}
Let $G$ be a linear algebraic group, and let $T \subset G$ be a maximal torus. An algebraic subgroup $U \subset G$, isomorphic to $\mathbb{G}_a$, is called a root subgroup with respect to $T$ if there exists a character $\XX^\alpha$ of $T$, referred to as the weight of $U$, such that
$$
t \circ \varepsilon(s) \circ t^{-1} = \varepsilon(\XX^\alpha(t) \cdot s) \quad \mbox{for all } t \in T \mbox{ and all } s \in \mathbb{G}_a\,,
$$
where $\varepsilon \colon \mathbb{G}_a \to U$ is a fixed isomorphism. The weight of a root subgroup $U$ does not depend on the choice of the isomorphism $\varepsilon$.
\end{definition}

With this definition, we arrive at the following well-known result. For the convenience of the reader, we provide a short proof.

\begin{proposition}
 Let $G$ be a connected linear algebraic group, and let $T$ be a maximal torus of $G$. Then $G$ is generated as a group by $T$ and all root subgroups.
\end{proposition}

\begin{proof}
Since the characteristic of $\KK$ is zero, there exists a Levi subgroup $L \subset G$, that is, a reductive subgroup such that $G = L \ltimes R$, where $R$ is the unipotent radical of $G$ \cite[Definition~11.22]{Bor91}. Furthermore, we can assume that $T \subset L$. Moreover, all root subgroups of $L$ with respect to $T$, together with $T$, generate $L$ \cite[Proposition~8.1.1.]{Spr09}. Since $R$ is normal in $G$, the torus $T$ acts on the Lie algebra of $R$ by conjugation, which then decomposes as a direct sum of root subspaces. The corresponding subgroups of $R$ are root subgroups that generate $R$. This completes the proof of the proposition.
\end{proof}
In the next lemma, we show that elements in root subgroups correspond to homogeneous nilpotent derivations.

\begin{lemma} \label{lem:unipotent-lnd}
Let $S$ be the first octant and let $I$ be a full monomial ideal of $\KK[S]$ with cofinite support. Suppose $g \in \Aut^0_\KK(\KK[S]/I)$ and $g$ is contained in a root subgroup with weight $\XX^\alpha$ with respect to $T$. Then $g = \exp \overline{\partial}$, where $\overline{\partial}$ is a nilpotent derivation of degree $\alpha$.
\end{lemma}

\begin{proof}
Since $g$ is contained in a root subgroup of $\Aut^0_\KK(\KK[S]/I)$, $g$ is unipotent, so the smallest algebraic subgroup $U$ containing $g$ is a root subgroup with weight $\XX^\alpha$. The injection $U \hookrightarrow \Aut^0_\KK(\KK[S]/I)$ induces a $\GA$-action on $\KK[S]/I$. By \eqref{eq:exponential}, we have $g = \exp(s \overline{\partial})$ for some nilpotent derivation $\overline{\partial} \colon \KK[S]/I \to \KK[S]/I$ and some $s \in \GA$. Since $s \overline{\partial}$ is again nilpotent, we can assume $g = \exp \overline{\partial}$ by replacing $\overline{\partial}$ with $s \overline{\partial}$. According to \cref{def:root-subgroup},
$$t \circ \exp(\overline{\partial}) \circ t^{-1} = \exp(\XX^\alpha(t) \overline{\partial}) \quad \text{for all} \quad t \in T \text{ and all } s \in \mathbb{G}_a.$$
Since conjugation commutes with the exponential, we obtain
\begin{align} \label{element-in-root}
 t \circ \overline{\partial} \circ t^{-1} = \XX^\alpha(t) \cdot \overline{\partial} \quad \text{for all} \quad t \in T.
\end{align}
Now, letting $m' \in S \setminus \supp(I)$, we write
$$\partial(\overline{\XX}^{m'}) = \sum_{m \in S \setminus \supp(I)} a_m \overline{\XX}^m.$$
Applying \eqref{element-in-root} to $\overline{\XX}^{m'}$, we get
$$t^{-1}(m') \sum_{m \in S \setminus \supp(I)} t(m) \cdot a_m \overline{\XX}^m = t(\alpha) \sum_{m \in S \setminus \supp(I)} a_m \overline{\XX}^m.$$
This equality can only hold if $a_m = 0$ for all but possibly one value of $m$. This implies that $\partial(\overline{\XX}^{m'}) = a_m \overline{\XX}^m$, and thus $\overline{\partial}$ is homogeneous.
\end{proof}
\begin{remark}
The full strength of \cref{main-classification} is best appreciated in the infinite-dimensional setting, where all three types of degrees described therein can appear. In the finite-dimensional case, only conditions (i) and (ii) can hold in \cref{main-classification}, since $(\ZZ_{\geq 0}\cdot\alpha)\cap \supp(I)$ is never empty for $\alpha \in S$, as any non-zero homogeneous derivation of non-zero degree is automatically nilpotent. Furthermore, when $S$ is the first octant, the liftable homogeneous nilpotent derivations classified in \cref{main-classification} and \cref{lemma: trivial-derivation} account for all homogeneous nilpotent derivations on $\KK[S]/I$ by \cref{th:surjective derivatio}.
\end{remark}
Let $S$ be the first octant and let $I$ be a full monomial ideal with cofinite support. We denote by $\mathcal{R}(S,I)$ the subset of $M$ consisting of $\alpha$ for which there exists a non-zero liftable homogeneous nilpotent derivation $\overline{\partial}_{\alpha,p}$. According to \cref{lem:unipotent-lnd}, these are precisely the weights of the non-trivial root subgroups of $\Aut^0_\KK(\KK[S]/I)$. Moreover, for each $\alpha \in \mathcal{R}(S,I)$, we define the $\KK$-vector space
$$G(\alpha) = \{ p \in N_\KK \mid \overline{\partial}_{\alpha,p}\in \Der(\KK[S]) \}\,,$$ 
and we let 
\begin{align} \label{eq:iota}
\iota_{\alpha} \colon G(\alpha) \to \Aut_\KK(\KK[S]/I) \quad \text{by} \quad p \mapsto \exp \overline{\partial}_{\alpha,p},\quad \mbox{for all } \alpha \in \mathcal{R}(S,I)\,.
\end{align}

Now, we can present our main result regarding the connected component of $\Aut_\KK(\KK[S]/I)$, which directly follows from \cref{lem:unipotent-lnd} and the aforementioned considerations.

\begin{proposition} \label{prop:aut0}
Let $S$ be the first octant and let $I$ be a full monomial ideal of $\KK[S]$ with cofinite support. Also, let $T = \spec \KK[M] \subset \Aut_\KK(\KK[S]/I)$. Then, the neutral component $\Aut^0_\KK(\KK[S]/I)$ is generated by $T$ and the images of $\iota_{\alpha}$ for all $\alpha \in \mathcal{R}(S,I)$. 
\end{proposition}

\begin{remark}
For each $\alpha\in \mathcal{R}(S,I)$, the set $L(\alpha) = \{\overline{\partial}_{\alpha,p}\mid p\in G(\alpha)\}$ is the subspace of nilpotent derivations in $\operatorname{Der}(\KK[S]/I)$ with degree $\alpha$. The dimension of $L(\alpha)$ may be any integer from $1$ to $n$. Indeed, in \cref{ex:dim_geq1} below, we provide an example where $\dim L(\alpha)=n$ and $L(\alpha)$ is moreover an abelian Lie subalgebra. This is in stark contrast with the case of complete toric variety, where all $ L(\alpha) $ are 1-dimensional \cite{demazure1970sous} (see also \cite[Theorem~2~$(i)$]{LL21}). 
\end{remark}

\begin{example} \label{ex:dim_geq1}
Let $M = \mathbb{Z}^n$ and $S = \mathbb{Z}^n_{\geq 0}$. Consider the ideal $I = (x_1^3,x_2^3\dots , x_n^3)$. By \cref{main-classification} and \cref{lemma: trivial-derivation}, we have that $\alpha = (1, 1,\dots ,1) \in \mathcal{R}(S, I)$ and $G(\alpha) = N_\KK$. Furthermore, every $p = (p_1, p_2,\dots,p_n) \in G(\alpha)$ defines the following derivation and corresponding exponential automorphism
$$
\begin{array}{ccc}
 \begin{array}{rcl}\overline{\partial}_{\alpha,p} \colon \KK[S]/I& \to & \KK[S]/I \\ 
 \x_i & \mapsto & p_i\x_1\dots \x_i^2\dots \x_n  \\ 
 \end{array}
&\quad\mbox{so that}\quad&
 \begin{array}{rcl}
 \operatorname{exp}(\overline{\partial}_{\alpha,p}) \colon \KK[S]/I& \to & \KK[S]/I \\ 
 \x_i & \mapsto & \x_i+p_i\x_1\dots \x_i^2\dots \x_n \\ 
 \end{array}\\
 \end{array}\,.
$$
The subspace $L(\alpha)=\{\overline{\partial}_{\alpha,p}\mid p\in N_\KK\}$ is $n$-dimensional in this case. Moreover, $L(\alpha)$ is an abelian Lie subalgebra, since the bracket $[\overline{\partial}_{\alpha,p}, \overline{\partial}_{\alpha,q}]$ has degree $2\alpha = (2,\dots,2)$, and for every $i$ we have $2\alpha + e_i = (2,\dots,3,\dots,2) \in \supp(I)$, so the bracket is the zero derivation on $\KK[S]/I$ by \cref{lemma: trivial-derivation}.
\end{example}

\begin{definition} 
Let $S$ and $S'$ be semigroups and let $I$ and $I'$ be monomial ideals of $\KK[S]$ and $\KK[S']$, respectively. A $\KK$-algebra homomorphism $g\colon \KK[S]/I \to \KK[S']/I'$ is called a toric morphism if there exists a semigroup homomorphism $\widehat{g}\colon S \to S'$ with $\widehat{g}(\supp(I)) \subset \supp(I')$ such that $g(\overline{\XX}^m) = \overline{\XX}^{\widehat{g}(m)}$. A toric automorphism of $\KK[S]/I$ is a toric endomorphism that admits an inverse which is also toric. We denote the set of toric automorphisms of $\KK[S]/I$ by $\Aut(S,I)$. An automorphism $g \in \Aut(S,I)$ corresponds to an automorphism $\widehat{g}\colon S \to S$ such that $\widehat{g}(\supp(I)) = \supp(I)$.
\end{definition}

\begin{lemma}\label{lemma:generator}
Let $S$ be a semigroup and let $I$ be a full monomial ideal of $\KK[S]$ with cofinite support. The group $\Aut(S,I)$ is finite.
\end{lemma}
\begin{proof}
Consider $g \in \Aut(S,I)$, which corresponds to the semigroup automorphism $\widehat{g}\colon S \to S$. The map $\widehat{g}$ induces a permutation of the set of rays $S(1)$. Consequently, we obtain a homomorphism from $\Aut(S,I)$ to the symmetric group on $S(1)$. Furthermore, since $S(1)$ spans $M_\RR$ as a vector space, the action of $\widehat{g}$ on $S(1)$ uniquely determines $\widehat{g}$, implying that the homomorphism is injective.
\end{proof}

\begin{proposition} \label{prop:finite}
Let $S$ be the first octant and let $I$ be a full monomial ideal of $\KK[S]$ with cofinite support. Then $\Aut_\KK(\KK[S]/I)/\Aut_\KK^0(\KK[S]/I)$ is a finite group generated by the images of toric morphisms.
\end{proposition}
\begin{proof}
Since $\Aut_\KK(\KK[S]/I)$ is a linear algebraic group by \cref{propositio: generating}, it follows that the quotient group $F=\Aut_\KK(\KK[S]/I)/\Aut_\KK^0(\KK[S]/I)$ is finite. Let $g \in \Aut_\KK(\KK[S]/I)$ and let $[g]$ be its image in $F$. As all maximal tori are conjugate in a linear algebraic group \cite[Theorem~6.3.5]{Spr09}, there exists $h \in \Aut_\KK^0(\KK[S]/I)$ such that 
$$(g \circ h)^{-1} \circ T \circ (g \circ h) = h^{-1} \circ (g^{-1} \circ T \circ g) \circ h = T\,.$$
Since $[g] = [g \circ h]$, we can replace $g$ with $g \circ h$, assuming that $g$ normalizes $T$.

Now, for $m' \in S \setminus \supp(I)$, we have $$g(\overline{\XX}^{m'}) = \sum_{m \in S \setminus \supp(I)} a_m \overline{\XX}^m\,.$$ 
The set $V = \KK \cdot \overline{\XX}^{m'} \subset \KK[S]/I$ is a 1-dimensional vector subspace of $\KK[S]/I$. Applying $g \circ T = T \circ g$ to $V$, we obtain
\begin{align*} 
 \KK \cdot \sum_{m \in S \setminus \supp(I)} a_m \overline{\XX}^m = \KK \cdot \left\{ \sum_{m \in S \setminus \supp(I)} t(m) a_m \overline{\XX}^m \mid t \in T \right\}\,.
\end{align*}
The left-hand side is a $\KK$-vector space, while the right-hand side is a $\KK$-vector space if and only if $a_m = 0$ for all but at most one $m \in S \setminus \supp(I)$. This implies that $g(\overline{\XX}^{m'}) = a_{m} \overline{\XX}^{m}$ for some $m \in S \setminus \supp(I)$.  Let now $m'= (m'_1, \ldots, m'_n)\neq 0$ and choose $1 \leq k \leq n$ such that $m'_k \neq 0$. For $t = (\lambda_1, \ldots, \lambda_n) \in T$ with $\lambda_i = 1$ for $i \neq k$ and $\lambda_k \in \KK^*\subset \KK$ a root of $p(r) = a_m r^{m'_k} - 1 \in \KK[r]$,
replacing $g$ by $g \circ t$, we obtain $[g] = [g \circ t]$ and $g(\overline{\XX}^{m'}) = \overline{\XX}^{m}$. Now, using the previous construction and the fact that $I$ is full, for each $e_i$ we fix $t_i \in T$ such that $g \circ t_i(\overline{\XX}^{e_i}) = \overline{\XX}^{m_i}$. Replacing $g$ by $g \circ t_n \circ \cdots \circ t_1$, we obtain $[g] = [g \circ t_n \circ \cdots \circ t_1]$ and $g(\overline{\XX}^{e_i}) = \overline{\XX}^{m_i}$ for all $i$. We can then define $\widehat{g} \colon S \to S$ by $\widehat{g}(e_i) = m_i$.  Since $g$ is an automorphism, $\widehat{g}$ is also an automorphism, and thus $g$ is a toric automorphism.
\end{proof}

For each $\alpha \in \mathcal{R}(S,I)$ and $p \in G(\alpha)$ such that $\overline{\partial}_{\alpha,p} \neq 0$, we denote by $U_{\alpha,p}$ the root subgroup $$U_{\alpha,p} = \{\exp(t\overline{\partial}_{\alpha,p}) \mid t \in \KK\} \subset \Aut^0_\KK(\KK[S]/I),.$$ 
If $\alpha \in M \setminus S$, then the space $G(\alpha)$ is 1-dimensional by \cref{th:surjective derivatio}, so for any two nonzero $p, q \in G(\alpha)$ we have $U_{\alpha,p} = U_{\alpha,q}$. In this case, we simply write $U_\alpha$.

\begin{lemma}\label{lemma: finite}
Let $S$ be the first octant and let $I\subset \KK[S]$ be a monomial ideal with $I$ full. Letting $i\neq j$, assume that the semigroup automorphism permuting $e_i$ with $e_j$ and leaving fixed $e_k$ for $k\neq i,j$ induces a toric automorphism of $\tau_{ij}\colon\KK[S]/I\to\KK[S]/I$.
Then $\tau_{ij}$ belongs to $\Aut^0_\KK(\KK[S]/I)$ if and only if $e_i-e_j$ and $e_j-e_i$ belong to $\mathcal{R}(S,I)$.
\end{lemma}

\begin{proof}
We first prove the forward implication. Assume $e_i-e_j, e_j-e_i\in \mathcal{R}(S,I)$. Let $\alpha = e_i - e_j$, $\beta = e_j - e_i$, $p = e_j^*$, and $q = e_i^*$. We define the automorphisms $g = \exp(\overline{\partial}_{\alpha,p})$ and $h = \exp(\overline{\partial}_{\beta,q})$. By the definition of homogeneous derivations, we have $\overline{\partial}_{\alpha,p}(\x_j) = \x_i$ and $\overline{\partial}_{\alpha,p}(\x_{j'}) = 0$ for $j' \neq j$, which gives $g(\x_j) = \x_j + \x_i$ and $g(\x_{j'}) = \x_{j'}$ for $j' \neq j$. Similarly, $\overline{\partial}_{\beta,q}(\x_i) = \x_j$ and $\overline{\partial}_{\beta,q}(\x_{i'}) = 0$ for $i' \neq i$, which gives $h(\x_i) = \x_i + \x_j$ and $h(\x_{i'}) = \x_{i'}$ for $i' \neq i$. 

For $t\in T$, our fixed diagonal torus, we define $f_t$ to be the automorphism determined by $t$. With $t=(t_1,\dots,t_n)$ where $t_j=-1$ and $t_k=1$ for $k\neq j$, we set $\varphi=f_t\circ g$ and $\psi=f_t\circ h$. By this construction, we have $\varphi(\x_i)=\x_i$ and $\varphi(\x_j)=\x_i-\x_j$, and $\psi(\x_i)=\x_i-\x_j$ and $\psi(\x_j)=-\x_j$. The transposition $\tau_{ij}$ that permutes $\x_i$ with $\x_j$ is given by $\tau_{ij}=\varphi^{-1}\circ\psi\circ\varphi$. This shows $\tau_{ij}\in\Aut^0_\KK(\KK[S]/I)$.

To prove the converse implication, suppose $\tau_{ij}\in\Aut^0_\KK(\KK[S]/I)$. Consider the Levi decomposition $\Aut^0_\KK(\KK[S]/I)=L\ltimes R$, with $R$ the unipotent radical and $L\supseteq T$ a reductive subgroup.

We now define a map
\begin{equation}\label{eq:Phi}
    \Phi\colon \Aut^{0}_\KK(\KK[S]/I)\to \operatorname{GL}(\overline{\mathfrak{m}}/\overline{\mathfrak{m}}^2)\cong \operatorname{GL}_n(\KK)
\end{equation}
as follows: given $g\in \Aut^{0}_\KK(\KK[S]/I)$, write $g(\overline{x}_i) = \sum_{k=1}^n a_{ik}\overline{x}_k + h_i$ with $h_i\in \overline{\mathfrak{m}}^2$, and define $\Phi(g)$ to be the matrix $(a_{ik})$. Here $\overline{\mathfrak m}=(\overline{x}_1,\dots,\overline{x}_n)$ denotes the maximal ideal of the local algebra $\KK[S]/I$, that is, the image of $\mathfrak m=(x_1,\dots,x_n)$ under the quotient map. This map is well-defined because the monomial ideal is full and contains no linear monomials, and the algebra $\KK[S]/I$ is local.

If $g\in\ker\Phi$, then $g$ acts as the identity on $\overline{\mathfrak m}/\overline{\mathfrak m}^2$ and hence, by multiplicativity, on each quotient $\overline{\mathfrak m}^{\,j}/\overline{\mathfrak m}^{\,j+1}$. Since $\overline{\mathfrak m}$ is nilpotent, the filtration $\KK[S]/I\supset\overline{\mathfrak m}\supset\overline{\mathfrak m}^2\supset\cdots$ is finite, so $g$ is unipotent. Therefore $\ker\Phi$ is a closed normal subgroup consisting of unipotent elements, and as $\operatorname{char}\KK=0$, it is contained in the unipotent radical $R$. Since moreover $L\cap R=\{\operatorname{Id}\}$, the restriction $\Phi|_L$ is injective, and being an injective homomorphism of algebraic groups in characteristic zero, it is a closed immersion; hence it embeds $L$ into $\operatorname{GL}_n(\KK)$, carrying $T$ onto the diagonal torus. Consequently every root of $L$ with respect to $T$ is of the form $e_{i'}-e_{j'}$ for some $i'\neq j'$. As $L$ is reductive, if $e_{i'}-e_{j'}$ is a root of $L$ then so is its opposite $e_{j'}-e_{i'}$; in particular both $e_{i'}-e_{j'}$ and $e_{j'}-e_{i'}$ are nonzero weights of $T$ on $\operatorname{Lie}\Aut^0_\KK(\KK[S]/I)=\Der(\KK[S]/I)$. Each such weight $\alpha$ is the degree of a nonzero homogeneous derivation $\overline{\partial}$ of $\KK[S]/I$; this $\overline{\partial}$ is nilpotent, since $\KK[S]/I$ is finite-dimensional and $\alpha\neq 0$, and liftable by \cref{th:surjective derivatio} as $S$ is the first octant, so $\overline{\partial}=\overline{\partial}_{\alpha,p}$ for some $p$. By the definition of $\mathcal{R}(S,I)$, both $e_{i'}-e_{j'}$ and $e_{j'}-e_{i'}$ belong to $\mathcal{R}(S,I)$.

For $i'\neq j'$, write $i'\sim j'$ when both $e_{i'}-e_{j'}$ and $e_{j'}-e_{i'}$ belong to $\mathcal{R}(S,I)$. By \cref{main-classification}, $e_{i'}-e_{j'}\in\mathcal{R}(S,I)$ if and only if the derivation $x_{i'}\,\partial/\partial x_{j'}$ of $\KK[S]$, which sends $\XX^\gamma$ to $\gamma_{j'}\,\XX^{\gamma-e_{j'}+e_{i'}}$, preserves $I$ and induces a non-zero derivation on $\KK[S]/I$. Since $I$ is full, $\overline{x}_{i'}\neq 0$, so the latter holds automatically, and the former is equivalent to
\begin{equation}\label{eq:move-units}
\gamma\in\supp(I)\ \text{and}\ \gamma_{j'}\geq 1\ \Longrightarrow\ \gamma-e_{j'}+e_{i'}\in\supp(I).
\end{equation}
The relation $\sim$ is symmetric by definition, and we set $i\sim i$ for all $i$, so that $\sim$ is reflexive; it remains to verify transitivity. If two of $i',j',k$ coincide the claim is immediate, so suppose $i'\sim j'$ and $j'\sim k$ with $i',j',k$ pairwise distinct, and let $\gamma\in\supp(I)$ with $\gamma_{i'}\geq 1$. 
Since $e_{j'}-e_{i'}\in\mathcal{R}(S,I)$, the criterion \eqref{eq:move-units} gives $\gamma-e_{i'}+e_{j'}\in\supp(I)$. 
Set $\delta=\gamma-e_{i'}+e_{j'}$, and note that $\delta_{j'}=\gamma_{j'}+1\geq 1$. Since $e_{k}-e_{j'}\in\mathcal{R}(S,I)$, applying the same criterion to $\delta$ gives $\delta-e_{j'}+e_{k}=\gamma-e_{i'}+e_{k}\in\supp(I)$. 
Thus, for every $\gamma\in\supp(I)$ with $\gamma_{i'}\geq 1$, we have $\gamma-e_{i'}+e_{k}\in\supp(I)$, which by \eqref{eq:move-units} means $e_{k}-e_{i'}\in\mathcal{R}(S,I)$. The membership $e_{i'}-e_{k}\in\mathcal{R}(S,I)$ follows by the symmetric argument. Hence $i'\sim k$, so $\sim$ is an equivalence relation, and its equivalence classes $B_1,\dots,B_r$ partition $\{1,\dots,n\}$.

Let $W=N_{\Aut^0_\KK(\KK[S]/I)}(T)/T$ denote the Weyl group of $\Aut^0_\KK(\KK[S]/I)$, well-defined since the centralizer of $T$ equals $T$ by \cref{lemma: torus}. It acts faithfully on the character lattice $X^*(T)=\ZZ^n$ by $(w\cdot\chi)(t)=\chi(\dot{w}^{-1}t\dot{w})$, where $\dot{w}\in N_{\Aut^0_\KK(\KK[S]/I)}(T)$ is any representative of $w$ \cite[7.1.4]{Spr09}. The projection onto the Levi factor $\Aut^0_\KK(\KK[S]/I)=L\ltimes R\to L$ restricts to the identity on $T$ and induces an isomorphism $W\xrightarrow{\sim} W(L,T)$. Surjectivity is clear, and the projection is the identity on $T$; for injectivity, let $g\in N_{\Aut^0_\KK(\KK[S]/I)}(T)$ project into $T$ and write $g=\ell r$ with $\ell\in L$ and $r\in R$, so that $\ell\in T$ and hence $r=\ell^{-1}g\in R\cap N_{\Aut^0_\KK(\KK[S]/I)}(T)$. For every $t\in T$, the commutator $rtr^{-1}t^{-1}$ lies in $T$ (as $r$ normalizes $T$) and in $R$ (as $R$ is normal), so $rtr^{-1}t^{-1}\in T\cap R=\{\operatorname{Id}\}$; thus $r$ centralizes $T$, and by \cref{lemma: torus} we get $r\in T\cap R=\{\operatorname{Id}\}$, whence $g=\ell\in T$. This isomorphism, by \cite[Theorem~7.1.9]{Spr09} is generated by the reflections $s_\beta$ associated with the roots $\beta$ of $L$. By the above, every such $\beta$ has the form $e_{i'}-e_{j'}$ with $i'\sim j'$, and $s_{e_{i'}-e_{j'}}$ interchanges the coordinates $i'$ and $j'$, acting on $X^*(T)$ as the transposition $(i'\,j')$. Therefore, inside the symmetric group $S_n$,
$$W\subseteq\big\langle\,(i'\,j') : i'\sim j'\,\big\rangle=\prod_{a=1}^{r}\operatorname{Sym}(B_a).$$
Finally, as $\tau_{ij}\in\Aut^0_\KK(\KK[S]/I)$ normalizes $T$, it defines a class $[\tau_{ij}]\in W$ whose image under the embedding $W\hookrightarrow S_n$ given by the action on $X^*(T)$ is the transposition $(i\,j)$. Every element of $\prod_{a}\operatorname{Sym}(B_a)$ maps each index to one in the same class; since $(i\,j)$ sends $i$ to $j$ and lies in $\prod_a\operatorname{Sym}(B_a)$ by the inclusion $W\subseteq\prod_a\operatorname{Sym}(B_a)$, the indices $i$ and $j$ belong to the same class, i.e.\ $i\sim j$. Hence both $e_i-e_j$ and $e_j-e_i$ belong to $\mathcal{R}(S,I)$.

\end{proof}

\begin{proposition}\label{proposition: finite}
Let $S$ be the first octant and let $I\subset \KK[S]$ be a monomial ideal with $I$ full. Let $\tau\in \Aut(S,I)$ be a toric automorphism. Then $\tau\in \Aut^0_\KK(\KK[S]/I)$ if and only if $\tau$ is a product of toric transpositions $\tau_{i_1j_1}\cdots\tau_{i_kj_k}$ with $\tau_{i_\ell j_\ell}\in\Aut^0_\KK(\KK[S]/I)$ for all $\ell=1,\ldots,k$.
\end{proposition}

\begin{proof}

The converse implication is immediate, so we only prove the direct implication. Let $\tau\in \Aut(S,I)$ be a toric automorphism such that $\tau\in \Aut^{0}_\KK(\KK[S]/I)$. Being toric, $\tau$ permutes the variables: there is a permutation $\sigma\in S_n$ with $\tau(\overline{x}_i)=\overline{x}_{\sigma(i)}$ for all $i$, and $\sigma$ preserves $\supp(I)$.

Since $\tau$ permutes coordinates, it normalizes $T$; as $\tau\in \Aut^{0}_\KK(\KK[S]/I)$, it defines a class $[\tau]$ in the Weyl group $W$ of $\Aut^{0}_\KK(\KK[S]/I)$, acting on $X^*(T)=\ZZ^n$ as $\sigma$. Recall from the proof of \cref{lemma: finite} the equivalence relation $\sim$ on $\{1,\dots,n\}$, with classes $B_1,\dots,B_r$, where $i\sim j$ means that both $e_i-e_j$ and $e_j-e_i$ belong to $\mathcal{R}(S,I)$, together with the inclusion $W\subseteq\prod_{a=1}^{r}\operatorname{Sym}(B_a)$. Hence $\sigma=[\tau]\in\prod_a\operatorname{Sym}(B_a)$. Since each symmetric group $\operatorname{Sym}(B_a)$ is generated by the transpositions of pairs of elements of $B_a$, we may write
$$\sigma=(i_1\,j_1)\cdots(i_k\,j_k),\qquad\text{where } i_\ell\sim j_\ell \text{ for all }\ell\,,$$
each transposition exchanging two indices that lie in a common block.

Fix $\ell$. We first check that exchanging the coordinates $i_\ell$ and $j_\ell$ preserves $\supp(I)$, so that the toric transposition $\tau_{i_\ell j_\ell}$ belongs to $\Aut(S,I)$. Let $\gamma\in\supp(I)$ and assume $\gamma_{i_\ell}\geq\gamma_{j_\ell}$; the opposite case is symmetric, with the roles of $i_\ell$ and $j_\ell$ interchanged. Since $e_{j_\ell}-e_{i_\ell}\in\mathcal{R}(S,I)$, the criterion \eqref{eq:move-units} gives $\delta-e_{i_\ell}+e_{j_\ell}\in\supp(I)$ for every $\delta\in\supp(I)$ with $\delta_{i_\ell}\geq 1$. Applying this successively to $\gamma$, the exponents
$$\gamma-t\,(e_{i_\ell}-e_{j_\ell}),\qquad t=0,1,\dots,\gamma_{i_\ell}-\gamma_{j_\ell},$$
all lie in $\supp(I)$: each is obtained from the previous one via \eqref{eq:move-units}, which applies because the $i_\ell$-th coordinate $\gamma_{i_\ell}-t$ satisfies $\gamma_{i_\ell}-t\geq\gamma_{j_\ell}+1\geq 1$ for $t<\gamma_{i_\ell}-\gamma_{j_\ell}$. The exponent for $t=\gamma_{i_\ell}-\gamma_{j_\ell}$ is exactly $\gamma$ with its $i_\ell$-th and $j_\ell$-th coordinates exchanged, so $\tau_{i_\ell j_\ell}$ preserves $\supp(I)$. By \cref{lemma: finite}, $\tau_{i_\ell j_\ell}\in\Aut^{0}_\KK(\KK[S]/I)$.

Finally, $\tau$ and $\tau_{i_1 j_1}\cdots\tau_{i_k j_k}$ are toric automorphisms inducing the same permutation $\sigma$ of the variables, hence they coincide. Therefore $\tau=\tau_{i_1 j_1}\cdots\tau_{i_k j_k}$ is a product of toric transpositions lying in $\Aut^{0}_\KK(\KK[S]/I)$.
\end{proof}

\begin{remark}\label{remark: finite}
The previous result proves that to study when a toric automorphism lies in the identity component, we need to determine when $e_i-e_j$ and $e_j-e_i$ are Demazure roots. Equivalently, this amounts to determining when $x_i\frac{\partial}{\partial x_j}$ and $x_j\frac{\partial}{\partial x_i}$ both descend to derivations of the quotient ring $\KK[S]/I$. That is, if $\XX^{\alpha}$ is a generator of $I$ with $\XX^{e_i}$ as a factor, then $\XX^{\alpha+e_j-e_i}\in I$, and if $\XX^{\beta}$ is a generator of $I$ with $\XX^{e_j}$ as a factor, then $\XX^{\beta+e_i-e_j}\in I$.
\end{remark}
\begin{example}\label{ex:non-connected}
Let $S = \ZZ^3_{\geq 0}$ and consider the monomial ideal
$$I = (x_1^2x_2, x_2^2x_3, x_3^2x_1, x_1^3, x_2^3, x_3^3) \subset \KK[x_1, x_2, x_3].$$
The ideal $I$ is invariant under the cyclic permutations $(123)$ and $(132)$, but no transposition preserves $I$. By \cref{proposition: finite}, any toric automorphism in $\Aut^0_\KK(\KK[S]/I)$ must be a product of transpositions in $\Aut^0_\KK(\KK[S]/I)$. Since no transpositions preserve $I$, the cyclic permutations are not in $\Aut^0_\KK(\KK[S]/I)$, and therefore $\Aut_\KK(\KK[S]/I)$ is not connected. Moreover, the component group $\Aut_\KK(\KK[S]/I)/\Aut^0_\KK(\KK[S]/I)$ is isomorphic to $\ZZ/3\ZZ$, generated by the image of $(123)$.
\end{example}

We can now present our main theorem, which summarizes all the results in this section and provides a description of the automorphism group of monomial algebras $\KK[S]/I$ in the case where $S$ is the first octant and $I$ has cofinite support. Recall that the map $\iota_\alpha$ is defined in \eqref{eq:iota}.

\begin{theorem} \label{th:aut}
Let $S$ be the first octant and let $I$ be a full monomial ideal of $\KK[S]$ with cofinite support. Also, let $T = \spec \KK[M] \subset \Aut_\KK(\KK[S]/I)$.
\begin{enumerate}[$(i)$]
 \item The automorphism group $\Aut_\KK(\KK[S]/I)$ is linear algebraic, and $T$ is a maximal torus of $\Aut_\KK(\KK[S]/I)$.
 \item The neutral component $\Aut^0_\KK(\KK[S]/I)$ is generated by $T$ and the images of $\iota_\alpha$ for all $\alpha \in \mathcal{R}(S, I)$.
 \item The quotient group $\Aut_\KK(\KK[S]/I) / \Aut^0_\KK(\KK[S]/I)$ is generated by the images of toric automorphisms.
 \end{enumerate}
\end{theorem}

\begin{proof}
Statement $(i)$ is contained in \cref{propositio: generating} and \cref{lemma: torus}. Statement $(ii)$ is in \cref{prop:aut0}. Finally, statement $(iii)$ is in \cref{prop:finite}.
\end{proof}

\begin{proposition}\label{cor:solvable}
The connected component $\Aut^0_\KK(\KK[S]/I)$ is solvable if and only if there is no pair $i\neq j$ such that both $e_i-e_j$ and $e_j-e_i$ belong to $\mathcal{R}(S,I)$. Equivalently, this is the condition for the reductive part of $\Aut^0_\KK(\KK[S]/I)$ to be exactly the torus $T$.
\end{proposition}

\begin{proof}
By the Levi decomposition, $\Aut^0_\KK(\KK[S]/I) = L \ltimes R$, where $L$ is the reductive part and $R$ is the unipotent radical. In $L$, root subgroups come in pairs corresponding to roots $\alpha$ and $-\alpha$; via the map $\Phi$ defined in \cref{lemma: finite}, the reductive part embeds into $\operatorname{Im}\Phi$, where only roots of the form $\alpha = e_i - e_j$ with $i \neq j$ can appear. Therefore, if both $e_i-e_j$ and $e_j-e_i$ belong to $\mathcal{R}(S,I)$, the corresponding root subgroups $U_{e_i-e_j}$ and $U_{e_j-e_i}$ together 
with $T$ generate a subgroup of $L$ 
containing a subgroup isomorphic to $\operatorname{GL}_2(\KK)$, 
which is not solvable. Conversely, if no such pair exists, every root subgroup lies in $R$, and by \cref{lemma: finite} no toric transposition can lie in $\Aut^0_\KK(\KK[S]/I)$, forcing $L = T$. In this case, $\Aut^0_\KK(\KK[S]/I) = T \ltimes R$ is solvable, since both $T$ and $R$ are solvable.
\end{proof}

We finish the paper providing several examples illustrating our results. Let $d = \dim_\KK(\KK[S]/I)$. Note that since automorphisms fix the unit and preserve the maximal ideal $\overline{\mathfrak{m}}$, one could represent the action using $(d-1) \times (d-1)$-matrices by restricting to $\overline{\mathfrak{m}}$. However, for clarity and ease of presentation, we display the full $d \times d$-matrices showing the complete action on $\KK[S]/I$.

\begin{example}\label{example: x^4}

Letting $M=\ZZ$ and $S=\ZZ_{\geq0}$, we let $I=(x^4)$. The group of toric automorphisms is trivial so $$\Aut_\KK(\KK[S]/I)=\Aut^0_\KK(\KK[S]/I)=\Aut_\KK(\KK[x]/(x^4))\,.$$
By \cref{main-classification} and \cref{lemma: trivial-derivation}, we have that only
$\alpha_1=1$ and $\alpha_2=2$ give rise to homogeneous nilpotent derivations that are non-zero on $\KK[S]/I$. For $p\in N_\KK=G(\alpha_i)=\KK$, for $i=1,2$, we obtain the derivations
$$
\begin{array}{ccc}
 \begin{array}{rcl}\overline{\partial}_{\alpha_1,p} \colon \KK[S]/I& \to & \KK[S]/I \\ 
 \x & \mapsto & p\x^2 \\ 
 \end{array}
&
\quad \mbox{and}\quad 
&\begin{array}{rcl}\overline{\partial}_{\alpha_2,p} \colon \KK[S]/I& \to & \KK[S]/I \\ 
 \x & \mapsto & p\x^3 \\ 
 \end{array}\\

\end{array}\,.
$$
Applying the exponential map, we obtain 
$$
\begin{array}{ccc}
 \begin{array}{rcl}\exp(\overline{\partial}_{\alpha_1,p}) \colon \KK[S]/I& \to & \KK[S]/I \\ 
 \x & \mapsto & \x+p\x^2+p^2\x^3 \\ 
 \end{array}
&
\quad \mbox{and}\quad 
&\begin{array}{rcl}\exp(\overline{\partial}_{\alpha_2,p}) \colon \KK[S]/I& \to & \KK[S]/I \\ 
 \x & \mapsto & \x+p\x^3 \\ 
 \end{array}\\
\end{array}\,.
$$

Letting $e_1=1, e_2=\overline{x},e_3=\overline{x}^2, e_4=\overline{x}^3$ we have that $\KK[S]/I=\KK\cdot e_1\oplus\KK\cdot e_2\oplus\KK\cdot e_3\oplus\KK\cdot e_4$. Moreover, $\Aut_\KK(\KK[S]/I)\subset\operatorname{GL}_4(\KK)$ by the proof of \cref{propositio: generating}. And by \cref{th:aut}, the automorphism group $\Aut_\KK(\KK[S]/I)$ is generated by the matrices
$$
\begin{array}{cccc}
 
\begin{pmatrix}
1 & 0 & 0&0\\
0 & t & 0&0\\
0 & 0 & t^2&0\\
0 & 0 & 0&t^3\\
\end{pmatrix},
&\begin{pmatrix}
1 & 0 & 0&0\\
0 & 1 & 0&0\\
0 & r & 1&0\\
0 & r^2 & 2r&1\\
\end{pmatrix},&\mbox{and}&\begin{pmatrix}
1 & 0 & 0&0\\
0 & 1 & 0&0\\
0 & 0 & 1&0\\
0 & s & 0&1\\
\end{pmatrix},
\end{array}
$$
 with $t\in \GM$, $r\in \GA$ and $s\in \GA$.
\end{example}

\begin{example}\label{example: x^2y^2}

Letting $M=\ZZ^2$ and $S=\ZZ_{\geq0}^2$ we let $I=(x^2,y^2)$. The group of toric automorphisms is $F\cong \ZZ/2\ZZ$ with generator given by the permutation $\overline{x}\mapsto \overline{y}$ and $\overline{y}\to \overline{x}$. Hence, $\Aut_\KK(\KK[S]/I)$ is generated by $\Aut_\KK^0(\KK[S]/I)$ and $F$. 

\begin{figure}[H]
\begin{picture}(100,95)
\definecolor{gray1}{gray}{0.7}
\definecolor{gray2}{gray}{0.85}
\definecolor{green}{RGB}{0,124,0}
\textcolor{gray2}{\put(0,20){\vector(1,0){90}}}
\textcolor{gray2}{\put(10,10){\vector(0,1){88}}}

\put(0,10){\textcolor{gray1}{\circle*{3}}}
\put(10,10){\textcolor{gray1}{\circle*{3}}}
\put(20,10){\textcolor{gray1}{\circle*{3}}} 
\put(30,10){\textcolor{red}{\circle*{3}}} 
\put(40,10){\textcolor{red}{\circle*{3}}} 
\put(50,10){\textcolor{red}{\circle*{3}}} 
\put(60,10){\textcolor{red}{\circle*{3}}} 
\put(70,10){\textcolor{red}{\circle*{3}}} 
\put(80,10){\textcolor{red}{\circle*{3}}}

\put(0,20){\textcolor{gray1}{\circle*{3}}}
\put(10,20){\circle{3}} 
\put(20,20){\textcolor{green}{\circle*{3}}} 
\put(30,20){{\circle*{3}}} 
\put(40,20){{\circle*{3}}} 
\put(50,20){{\circle*{3}}} 
\put(60,20){\circle*{3}}
\put(70,20){\circle*{3}}
\put(80,20){\circle*{3}}

\put(0,30){\textcolor{gray1}{\circle*{3}}}
\put(10,30){\textcolor{green}{\circle*{3}}} 
\put(20,30){\textcolor{green}{\circle*{3}}} 
\put(30,30){{\circle*{3}}} 
\put(40,30){{\circle*{3}}} 
\put(50,30){{\circle*{3}}} 
\put(60,30){{\circle*{3}}} 
\put(70,30){\circle*{3}}
\put(80,30){\circle*{3}}

\put(0,40){\textcolor{red}{\circle*{3}}}
\put(10,40){{\circle*{3}}} 
\put(20,40){{\circle*{3}}} 
\put(30,40){{\circle*{3}}} 
\put(40,40){{\circle*{3}}}
\put(50,40){{\circle*{3}}}
\put(60,40){\circle*{3}} 
\put(70,40){\circle*{3}}
\put(80,40){\circle*{3}}

\put(0,50){\textcolor{red}{\circle*{3}}}
\put(10,50){{\circle*{3}}} 
\put(20,50){{\circle*{3}}} 
\put(30,50){{\circle*{3}}} 
\put(40,50){\circle*{3}}
\put(50,50){\circle*{3}}
\put(60,50){\circle*{3}} 
\put(70,50){\circle*{3}} 
\put(80,50){\circle*{3}}

\put(0,60){\textcolor{red}{\circle*{3}}}
\put(10,60){{\circle*{3}}} 
\put(20,60){{\circle*{3}}} 
\put(30,60){{\circle*{3}}} 
\put(40,60){\circle*{3}}
\put(50,60){\circle*{3}}
\put(60,60){\circle*{3}} 
\put(70,60){\circle*{3}}
\put(80,60){\circle*{3}}

\put(0,70){\textcolor{red}{\circle*{3}}}
\put(10,70){{\circle*{3}}}
\put(20,70){{\circle*{3}}}
\put(30,70){{\circle*{3}}}
\put(40,70){{\circle*{3}}}
\put(50,70){\circle*{3}}
\put(60,70){\circle*{3}} 
\put(70,70){\circle*{3}} 
\put(80,70){\circle*{3}}

\put(0,80){\textcolor{red}{\circle*{3}}}
\put(10,80){{\circle*{3}}}
\put(20,80){{\circle*{3}}}
\put(30,80){{\circle*{3}}}
\put(40,80){{\circle*{3}}}
\put(50,80){{\circle*{3}}}
\put(60,80){\circle*{3}} 
\put(70,80){\circle*{3}} 
\put(80,80){\circle*{3}}

\put(0,90){\textcolor{red}{\circle*{3}}}
\put(10,90){{\circle*{3}}}
\put(20,90){{\circle*{3}}}
\put(30,90){{\circle*{3}}}
\put(40,90){{\circle*{3}}}
\put(50,90){{\circle*{3}}}
\put(60,90){{\circle*{3}}}
\put(70,90){\circle*{3}} 
\put(80,90){\circle*{3}}

\end{picture}
\caption{$I=(x^2,y^2)$}\label{fig:1-ex5}
\end{figure}

By \cref{main-classification} and \cref{lemma: trivial-derivation}, we have that only $\alpha_1=(1,0)$ and $\alpha_2=(0,1)$ give rise to non-zero homogeneous nilpotent derivations with $G(\alpha_1)/\ker\iota_{\alpha_1}\cong\KK\cdot (0,1)$ and $G(\alpha_2)/\ker\iota_{\alpha_2}\cong\KK\cdot (1,0)$. Letting $p=(p_1,p_2)\in N_\KK$, this yields 
$$
\begin{array}{cc}
 \begin{array}{rcl}\overline{\partial}_{\alpha_1,p} \colon \KK[x,y]/I& \to & \KK[x,y]/I \\ 
 \x & \mapsto & 0 \\ 
 \y & \mapsto &p_2\x\y \\ 
 \end{array}

\quad \mbox{and}\quad 
&\begin{array}{rcl}\overline{\partial}_{\alpha_2,p} \colon \KK[x,y]/I& \to & \KK[x,y]/I \\ 
 \x & \mapsto & p_1\x\y \\ 
 \y & \mapsto & 0 \\ 
 \end{array}\\

\end{array}\,.
$$
Applying the exponential map, we obtain 
$$
\begin{array}{cc}
 \begin{array}{rcl}\exp(\overline{\partial}_{\alpha_1,p}) \colon \KK[x,y]/I& \to & \KK[x,y]/I \\ 
 \x & \mapsto & \x \\ 
 \y & \mapsto & \y+p_2\x\y \\ 
 \end{array}

\quad \mbox{and}\quad 
&\begin{array}{rcl}\exp(\overline{\partial}_{\alpha_2,p}) \colon \KK[x,y]/I& \to & \KK[x,y]/I \\ 
 \x & \mapsto & \x+p_1\x\y \\ 
 \y & \mapsto & \y \\ 
 \end{array}\\

\end{array}\,.
$$

Letting $e_1=1, e_2=\overline{y},e_3=\overline{x}, e_4=\overline{x}\overline{y}$ we have that $\KK[S]/I=\KK\cdot e_1\oplus\KK\cdot e_2\oplus\KK\cdot e_3\oplus\KK\cdot e_4$. Moreover $\Aut_\KK(\KK[x,y]/I)\subset\operatorname{GL}_4(\KK)$ by the proof of \cref{propositio: generating}, and by \cref{th:aut}, $\Aut_\KK^0(\KK[S]/I)$ is generated by the matrices
$$
\begin{array}{ccc}
 
\begin{pmatrix}
1 & 0 & 0&0\\
0 & t_2 & 0&0\\
0 & 0 & t_1&0\\
0 & 0 & 0&t_1t_2\\
\end{pmatrix},
&\begin{pmatrix}
1 & 0 & 0&0\\
0 & 1 & 0&0\\
0 & 0 & 1&0\\
0 & r & 0&1\\
\end{pmatrix},\quad \mbox{and}\quad &\begin{pmatrix}
1 & 0 & 0&0\\
0 & 1 & 0&0\\
0 & 0 & 1&0\\
0 & 0 & s&1\\
\end{pmatrix},
\end{array}
$$
 with $(t_1,t_2)\in T=\GM^2$, $r\in \GA$ and $s\in \GA$.

Finally, all matrices are lower triangular, so the generator of $F$ is not contained in $\Aut_\KK^0(\KK[S]/I)$. We conclude that 
$$\Aut_\KK(\KK[S]/I)\cong \Aut_\KK^0(\KK[S]/I)\rtimes F\,. $$
\end{example}

\begin{example}\label{example:non-solv}
Letting $M=\ZZ^2$ and $S=\ZZ_{\geq0}^2$ we let $I=(x^2,xy,y^2)$.

\begin{figure}[H]
\begin{picture}(100,95)
\definecolor{gray1}{gray}{0.7}
\definecolor{gray2}{gray}{0.85}
\definecolor{green}{RGB}{0,124,0}

\textcolor{gray2}{\put(0,20){\vector(1,0){90}}}

\textcolor{gray2}{\put(10,10){\vector(0,1){88}}}

\put(0,10){\textcolor{gray1}{\circle*{3}}}
\put(10,10){\textcolor{gray1}{\circle*{3}}}
\put(20,10){\textcolor{red}{\circle*{3}}} 
\put(30,10){\textcolor{red}{\circle*{3}}} 
\put(40,10){\textcolor{red}{\circle*{3}}} 
\put(50,10){\textcolor{red}{\circle*{3}}} 
\put(60,10){\textcolor{red}{\circle*{3}}} 
\put(70,10){\textcolor{red}{\circle*{3}}} 
\put(80,10){\textcolor{red}{\circle*{3}}}

\put(0,20){\textcolor{gray1}{\circle*{3}}}
\put(10,20){\circle{3}} 
\put(20,20){\textcolor{green}{\circle*{3}}} 
\put(30,20){{\circle*{3}}} 
\put(40,20){{\circle*{3}}} 
\put(50,20){{\circle*{3}}} 
\put(60,20){\circle*{3}}
\put(70,20){\circle*{3}}
\put(80,20){\circle*{3}}

\put(0,30){\textcolor{red}{\circle*{3}}}
\put(10,30){\textcolor{green}{\circle*{3}}} 
\put(20,30){{\circle*{3}}} 
\put(30,30){{\circle*{3}}} 
\put(40,30){{\circle*{3}}} 
\put(50,30){{\circle*{3}}} 
\put(60,30){{\circle*{3}}} 
\put(70,30){\circle*{3}}
\put(80,30){\circle*{3}}

\put(0,40){\textcolor{red}{\circle*{3}}}
\put(10,40){{\circle*{3}}} 
\put(20,40){{\circle*{3}}} 
\put(30,40){{\circle*{3}}} 
\put(40,40){{\circle*{3}}}
\put(50,40){{\circle*{3}}}
\put(60,40){\circle*{3}} 
\put(70,40){\circle*{3}}
\put(80,40){\circle*{3}}

\put(0,50){\textcolor{red}{\circle*{3}}}
\put(10,50){{\circle*{3}}} 
\put(20,50){{\circle*{3}}} 
\put(30,50){{\circle*{3}}} 
\put(40,50){\circle*{3}}
\put(50,50){\circle*{3}}
\put(60,50){\circle*{3}} 
\put(70,50){\circle*{3}} 
\put(80,50){\circle*{3}}

\put(0,60){\textcolor{red}{\circle*{3}}}
\put(10,60){{\circle*{3}}} 
\put(20,60){{\circle*{3}}} 
\put(30,60){{\circle*{3}}} 
\put(40,60){\circle*{3}}
\put(50,60){\circle*{3}}
\put(60,60){\circle*{3}} 
\put(70,60){\circle*{3}}
\put(80,60){\circle*{3}}

\put(0,70){\textcolor{red}{\circle*{3}}}
\put(10,70){{\circle*{3}}}
\put(20,70){{\circle*{3}}}
\put(30,70){{\circle*{3}}}
\put(40,70){{\circle*{3}}}
\put(50,70){\circle*{3}}
\put(60,70){\circle*{3}} 
\put(70,70){\circle*{3}} 
\put(80,70){\circle*{3}}

\put(0,80){\textcolor{red}{\circle*{3}}}
\put(10,80){{\circle*{3}}}
\put(20,80){{\circle*{3}}}
\put(30,80){{\circle*{3}}}
\put(40,80){{\circle*{3}}}
\put(50,80){{\circle*{3}}}
\put(60,80){\circle*{3}} 
\put(70,80){\circle*{3}} 
\put(80,80){\circle*{3}}

\put(0,90){\textcolor{red}{\circle*{3}}}
\put(10,90){{\circle*{3}}}
\put(20,90){{\circle*{3}}}
\put(30,90){{\circle*{3}}}
\put(40,90){{\circle*{3}}}
\put(50,90){{\circle*{3}}}
\put(60,90){{\circle*{3}}}
\put(70,90){\circle*{3}} 
\put(80,90){\circle*{3}}

\end{picture}\\
\caption{$I=(x^2,xy,y^2)$}\label{fig:1-ex6}
\end{figure}

The group of toric automorphisms $F$ is again isomorphic to $\ZZ/2\ZZ$ with generator given by the permutation $\overline{x}\mapsto \overline{y}$ and $\overline{y}\mapsto \overline{x}$. In this case, as we will show by the end of this example that $F\subset \Aut_\KK^0(\KK[S]/I)$. Hence,
$$\Aut_\KK(\KK[S]/I)= \Aut_\KK^0(\KK[S]/I). $$
By \cref{main-classification} and \cref{lemma: trivial-derivation}, we have that only $\alpha_1=(1,-1)$ and $\alpha_2=(-1,1)$ give rise to non-zero homogeneous nilpotent derivations with $G(\alpha_1)=\KK\cdot (0,1)$ and $G(\alpha_2)=\KK\cdot (1,0)$. Letting $p=(p_1,p_2)\in N_{\KK}$, this yields
$$
\begin{array}{cc}
 \begin{array}{rcl}\overline{\partial}_{\alpha_1,p} \colon \KK[x,y]/I& \to & \KK[x,y]/I \\ 
 \x & \mapsto & 0 \\ 
 \y & \mapsto & p_2\x \\ 
 \end{array}

\quad \mbox{and}\quad
&\begin{array}{rcl}\overline{\partial}_{\alpha_2,p} \colon \KK[x,y]/I& \to & \KK[x,y]/I \\ 
 \x & \mapsto & p_1\y \\ 
 \y & \mapsto & 0 \\ 
 \end{array}\\

\end{array}\,.
$$
Applying the exponential map, we obtain: 

$$
\begin{array}{cc}
 \begin{array}{rcl}\exp(\overline{\partial}_{\alpha_1,p}) \colon \KK[x,y]/I& \to & \KK[x,y]/I \\ 
 \x & \mapsto & \x \\ 
 \y & \mapsto & \y+p_2\x \\ 
 \end{array}

\quad \mbox{and}\quad
&\begin{array}{rcl}\exp(\overline{\partial}_{\alpha_2,p}) \colon \KK[x,y]/I& \to & \KK[x,y]/I \\ 
 \x & \mapsto & \x+p_1\y \\ 
 \y & \mapsto & \y \\ 
 \end{array}\\

\end{array}\,.
$$

Letting now $e_1=1, e_2=\overline{y}$ and $e_3=\overline{x}$ we have that $\KK[S]/I=\KK\cdot e_1\oplus\KK\cdot e_2\oplus\KK\cdot e_3$. Moreover $\Aut_\KK(\KK[x,y]/I)\subset\operatorname{GL}_3(\KK)$ by the proof of \cref{propositio: generating}. And by \cref{th:aut}, $\Aut_\KK^0(\KK[S]/I)$ is generated by the matrices
$$
\begin{array}{ccc}
 
\begin{pmatrix}
1 & 0 & 0\\
0 & t_2 & 0\\
0 & 0 & t_1\\

\end{pmatrix},
&\begin{pmatrix}
1 & 0 & 0\\
0 & 1 & 0\\
0 & r & 1\\

\end{pmatrix}, \quad \mbox{and}\quad &\begin{pmatrix}
1 & 0 & 0\\
0 & 1 & s\\
0 & 0 & 1\\
\end{pmatrix},
\end{array}
$$
with $(t_1,t_2)\in T=\GM^2$, $r\in \GA$ and $s\in \GA$. The group $\Aut_\KK^0(\KK[S]/I)$ is isomorphic to $\operatorname{GL}_2(\KK)$ taking the lower right $2\times 2$ blocks of the above matrices. 
 
In this case, $F\subset \Aut_\KK^0(\KK[S]/I)$ since the permutation $\x\mapsto \y$, $\y\mapsto \x$ is obtained as $\varphi^{-1}\circ\psi \circ \varphi$ where $\varphi,\psi\in \Aut_\KK^0(\KK[S]/I)$ are defined via 
$\varphi(\x,\y)=(\x,\x-\y)$ and $\psi(\x,\y)=(\x-\y,-\y)$.  This phenomenon where some toric morphisms are contained in $\Aut_\KK^0(\KK[S]/I)$ can only arise when there exist non-trivial root subgroups with weights $\alpha$ and $-\alpha$. 
\end{example}

\begin{remark} \label{rem:non-solv}
Note that $\Aut_\KK(\KK[S]/I)$ in \cref{example:non-solv} is not solvable since it contains $\operatorname{GL}_2(\KK)$. A similar result can be obtained for higher dimensions by taking $S=\ZZ_{\geq 0}^n$ and letting $I$ be the monomial ideal generated by all the monomials of degree $d\geq 2$. In this case, $\Aut_\KK(\KK[S]/I)$ is $\operatorname{GL}_n(\KK)$, and thus $\Aut_\KK(\KK[S]/I)$ is also not solvable.
\end{remark}

\begin{remark}[On Christophersen's problem]
Christophersen's problem asks whether for every finite-dimensional local algebra 
$A$ of dimension $d$ one has 
\[
\dim \Aut(A) \ge d-1,
\]
and whether equality characterizes $A \cong \KK[t]/(t^d)$; see \cite{Sta25}. 
In the monomial setting considered here, \cref{th:aut} provides an explicit 
description of $\Aut(A)$, making it natural to investigate this inequality 
within this class of algebras. 
\end{remark}

\bibliographystyle{alpha}
\bibliography{ref}

\end{document}